\documentclass[11pt,reqno]{amsart}

\usepackage{amssymb, amsmath, amsthm}
\usepackage{hyperref}
\usepackage[alphabetic,backrefs,lite]{amsrefs}
\usepackage[alphabetic,lite]{amsrefs}
\usepackage{verbatim}
\usepackage{amscd}   
\usepackage[all]{xy} 
\usepackage{youngtab}
\usepackage{young} 
\usepackage{ytableau}
\usepackage{tikz}
\usepackage{ mathrsfs }
\usepackage{cases}
\usepackage{array}
\usepackage{cellspace}
\usepackage{tabu}
\usepackage{calligra,mathrsfs}

\newcommand{\defi}[1]{{\upshape\sffamily #1}}
\DeclareMathOperator{\ShHom}{\mathscr{H}\text{\kern -3pt {\calligra\large om}}\,}

\renewcommand{\a}{\alpha}
\renewcommand{\b}{\beta}

\newcommand{\bw}{\bigwedge}
\renewcommand{\det}{\textrm{det}}

\renewcommand{\ll}{\lambda}

\newcommand{\onto}{\twoheadrightarrow}
\newcommand{\oo}{\otimes}
\newcommand{\pd}{\partial}
\renewcommand{\P}{\mathcal{P}}

\renewcommand{\SS}{\mathbb{S}}

\newcommand{\GL}{\operatorname{GL}}
\newcommand{\Hom}{\operatorname{Hom}}

\newcommand{\rk}{\operatorname{rank}}

\newcommand{\Sym}{\operatorname{Sym}}
\newcommand{\Tor}{\operatorname{Tor}}

\newcommand{\Graph}{\operatorname{Graph}}
\newcommand{\Rees}{\operatorname{Rees}}

\renewcommand{\det}{\operatorname{det}}

\newcommand{\gr}{\operatorname{gr}}

\renewcommand{\ker}{\operatorname{ker}}

\newcommand{\bb}[1]{\mathbb{#1}}

\newcommand{\mc}[1]{\mathcal{#1}}
\newcommand{\mf}[1]{\mathfrak{#1}}
\newcommand{\ol}[1]{\overline{#1}}
\newcommand{\op}[1]{\operatorname{#1}}

\newcommand{\ul}[1]{\underline{#1}}

\def\PP{{\textbf P}}
\def\lra{\longrightarrow}

\def\llra{\longleftrightarrow}

\newtheorem{theorem}{Theorem}[section]
\newtheorem*{theorem*}{Theorem}
\newtheorem*{problem*}{Problem}
\newtheorem{lemma}[theorem]{Lemma}

\newtheorem{proposition}[theorem]{Proposition}
\newtheorem{corollary}[theorem]{Corollary}
\newtheorem*{corollary*}{Corollary}

\newtheorem*{main-thm*}{Main Theorem}
\newtheorem*{linear-resolutions*}{Theorem on Linear Resolutions}
\newtheorem*{regularity-powers*}{Theorem on Regularity}
\newtheorem*{injectivity-Ext*}{Theorem on Injectivity of Maps of Ext Modules}
\newtheorem*{Kodaira*}{Kodaira Vanishing for Determinantal Thickenings}

\theoremstyle{definition}

\newtheorem*{definition*}{Definition}
\newtheorem{example}[theorem]{Example}
\newtheorem{question}[theorem]{Question}

\theoremstyle{remark}
\newtheorem{remark}[theorem]{Remark}
\newtheorem*{remark*}{Remark}

\numberwithin{equation}{section}

\makeatletter
\@namedef{subjclassname@2020}{
  \textup{2020} Mathematics Subject Classification}
\makeatother

\begin{document}

\title{Relations between the $2\times 2$ minors of a generic matrix}

\author{Hang Huang}
\address{Department of Mathematics,  Texas A\&M University,  Mailstop 3368,  College Station, TX 77843, USA}
\email{hhuang@math.tamu.edu}

\author{Michael Perlman}
\address{Department of Mathematics and Statistics, Queen’s University, Kingston, ON K7L 3N6, Canada}
\email{mp174@queensu.ca}

\author{Claudia Polini}
\address{Department of Mathematics, University of Notre Dame, 255 Hurley, Notre Dame, IN 46556, USA}
\email{cpolini@nd.edu}

\author{Claudiu Raicu}
\address{Department of Mathematics, University of Notre Dame, 255 Hurley, Notre Dame, IN 46556, USA\newline
\indent Institute of Mathematics ``Simion Stoilow'' of the Romanian Academy}
\email{craicu@nd.edu}

\author{Alessio Sammartano}
\address{Dipartimento di Matematica, Politecnico di Milano, Via Bonardi 9, 20133, Milano, Italy}
\email{alessio.sammartano@polimi.it}


\date{\today}


\begin{abstract} 
 We prove the  case $t=2$ of a conjecture of Bruns--Conca--Varbaro, describing the minimal relations between the $t\times t$ minors of a generic matrix. Interpreting these relations as polynomial functors, and applying transpose duality as in the work of Sam--Snowden, this problem is equivalent to understanding the relations satisfied by $t\times t$ generalized permanents. Our proof follows by combining Koszul homology calculations on the minors side, with  a study of subspace varieties on the permanents side, and with the Kempf--Weyman technique (on both sides).
\end{abstract}

\maketitle

\section{Introduction}\label{sec:intro}

For positive integers $m,n$, consider the rational map
\begin{equation}\label{eq:phi-wedgephi}
 \Lambda_2 : \bb{P}(\Hom(\bb{C}^m,\bb{C}^n)) \dashrightarrow \bb{P}\left(\Hom\left(\bw^2\bb{C}^m,\bw^2\bb{C}^n\right)\right),\quad \phi \lra \bw^2 \phi.
\end{equation}
We denote the (closure of its) image by $X_{m,n}$ and consider the problem of understanding the defining equations of $X_{m,n}$. Identifying $\Hom(\bb{C}^m,\bb{C}^n)$ with the space of $m\times n$ complex matrices, the map~(\ref{eq:phi-wedgephi}) simply assigns to a matrix the tuple consisting of all its $2\times 2$ minors, so finding the equations of $X_{m,n}$ amounts to understanding the algebraic relations that these minors satisfy. When $m=2$, we can identify $X_{2,n}$ with $\op{Gr}_2(\bb{C}^n)$, the \defi{Grassmannian} of $2$-dimensional subspaces of $\bb{C}^n$, whose equations are well-understood: they are all quadratic, known as the \defi{Pl\"ucker relations}.
For general $m,n$ however, the Pl\"ucker relations are not sufficient, as shown by Bruns, Conca and Varbaro in \cite{BCV}. 
In fact, \cite{BCV} identifies new minimal equations between the $t\times t$ minors for an arbitrary $t$, of degree 2 and 3.
They are described in  representation theoretic terms, and  obtained through a careful analysis of highest weight vectors and combinatorics of bi-tableaux.
It is conjectured in \cite[Conjecture~2.12]{BCV} that these quadratic and cubic equations generate all the relations between the minors.  The goal of our paper is to confirm the conjecture in the case when $t = 2$. 
This boils down to a vanishing result for $\Tor$ groups, which we prove by combining a number of techniques from representation theory and algebraic geometry.

To state the results, as well as for most proofs, it is convenient to use a coordinate independent approach, and to make the usual identification between matrices and $2$-tensors. To that end, we consider complex vector spaces $V_1,V_2$, with $\dim(V_1)=m$, $\dim(V_2)=n$. We let $S = \Sym(V_1\oo V_2)$, which we think of as the homogeneous coordinate ring of the source of  $\Lambda_2$, and we consider the natural action of the group $\GL=\GL(V_1)\times\GL(V_2)$ on $S$. If we identify $S\simeq\bb{C}[x_{i,j}]$ then the $2\times 2$ minors of the generic matrix of indeterminates $(x_{i,j})$ span a $\GL$-invariant subspace of $S$ isomorphic to $\bw^2 V_1\oo \bw^2 V_2$. We let $W=\bw^2 V_1\oo \bw^2 V_2$, and consider the polynomial ring $R = \Sym(W)$, which we think of as the homogeneous coordinate ring of the target of the map $\Lambda_2$. The inclusion of $W$ into $S$ gives rise to an algebra homomorphism $\Psi:R\to S$, whose image we denote by $A$. We have that $A=\bb{C}[X_{m,n}]$ is the homogeneous coordinate ring of $X_{m,n}$, and $I(X_{m,n}) = \ker(\Psi)$. Noting that the minimal generators of $I(X_{m,n})$ are encoded by $\Tor_1^R(A,\bb{C})$, we prove the following (here, for a partition $\ll=(\ll_1\geq\ll_2\geq\cdots)$, we write $\bb{S}_{\ll}$ for the corresponding \defi{Schur functor}, see Section~\ref{subsec:functors}).

\begin{theorem}\label{thm:relations-A}
 We have $\Tor_1^R(A,\bb{C})_j = 0$ for $j\neq 2,3$, and we have an isomorphism of $\GL$-representations
 \[ \Tor_1^R(A,\bb{C})_2 = \bb{S}_{1,1,1,1}V_1 \oo \bb{S}_{2,2}V_2 \oplus \bb{S}_{2,2}V_1 \oo \bb{S}_{1,1,1,1}V_2,\]
 \[ \Tor_1^R(A,\bb{C})_3 = \bb{S}_{3,1,1,1}V_1 \oo \bb{S}_{2,2,2}V_2 \oplus \bb{S}_{2,2,2}V_1 \oo \bb{S}_{3,1,1,1}V_2.\]
\end{theorem}

For our proof of Theorem~\ref{thm:relations-A}, we will consider in parallel the closely related problem of understanding relations between permanents. More precisely, we consider the map (which is defined everywhere)
\begin{equation}\label{eq:phi-symphi}
 \Sigma_2 : \bb{P}(\Hom(\bb{C}^m,\bb{C}^n)) \lra \bb{P}\left(\Hom\left(\Sym^2\bb{C}^m,\Sym^2\bb{C}^n\right)\right),\quad \phi \lra \Sym^2 \phi,
\end{equation}
and denote its image by $\ol{X}_{m,n}$. When $m=1$, $\ol{X}_{1,n}$ can be identified with the degree two Veronese variety, whose defining equations are again known to be quadratic. As we will see, $\ol{X}_{m,n}$ also admits cubic minimal relations in general. A \defi{generalized $2\times 2$ submatrix} is one of the form
\[\begin{bmatrix} x_{i_1,j_1} & x_{i_1,j_2} \\ x_{i_2,j_1} & x_{i_2,j_2} \end{bmatrix},\]
where we do not require that $i_1\neq i_2$, or that $j_1\neq j_2$. The corresponding \defi{generalized permanent} is given by $x_{i_1,j_1} \cdot x_{i_2,j_2} + x_{i_1,j_2} \cdot x_{i_2,j_1}$. The (generalized) $2\times 2$ permanents span a $\GL$-invariant subspace inside $S$, isomorphic to $\ol{W}=\Sym^2 V_1 \oo \Sym^2 V_2$, and complementary to the space of minors within the quadrics in $S$:
\[ S_2 = \Sym^2(V_1\oo V_2) = W \oplus \ol{W}.\]
We define $\ol{R}=\Sym(\ol{W})$, and let $\ol{A}$ be the image of the natural map $\ol{\Psi}:\ol{R}\lra S$ induced by the inclusion $\ol{W}\subset S$. The ring $\ol{A}$ is the algebra generated by the $2\times 2$ permanents, and is also the homogeneous coordinate ring of the image of $\Sigma_2$. Moreover, we have that $\ker(\ol{\Psi}) = I(\ol{X}_{m,n})$. We will prove the following, which in fact turns out to be equivalent to Theorem~\ref{thm:relations-A}.

\begin{theorem}\label{thm:relations-ol-A}
 We have $\Tor_1^{\ol{R}}(\ol{A},\bb{C})_j = 0$ for $j\neq 2,3$, and we have an isomorphism of $\GL$-representations
 \[ \Tor_1^{\ol{R}}(\ol{A},\bb{C})_2 = \bb{S}_{4}V_1 \oo \bb{S}_{2,2}V_2 \oplus \bb{S}_{2,2}V_1 \oo \bb{S}_{4}V_2,\]
 \[ \Tor_1^{\ol{R}}(\ol{A},\bb{C})_3 = \bb{S}_{4,1,1}V_1 \oo \bb{S}_{3,3}V_2 \oplus \bb{S}_{3,3}V_1 \oo \bb{S}_{4,1,1}V_2.\]
\end{theorem}

The relationship between Theorems~\ref{thm:relations-A} and~\ref{thm:relations-ol-A} comes from interpreting all the constructions described so far as polynomial functors, and using transpose duality as explained in Section~\ref{subsec:bi-functors}. For now, we note that it can be visualized by drawing the Young diagrams of the partitions associated with the relevant Schur functors. For instance, we have that the exterior and symmetric powers correspond to transposed diagrams:
\[\Yvcentermath1 \bw^2 \llra \yng(1,1) \qquad\mbox{ and }\qquad\Sym^2 \llra \yng(2)\]
More generally, the same is true if we compare the $\Tor$ groups for minors and permanents:
\[
\setlength{\extrarowheight}{3pt}
\ytableausetup{smalltableaux,aligntableaux=center}
\renewcommand{\arraystretch}{2}
\begin{array}{c|c|c}
 & \Tor_1^R(A,\bb{C}) & \Tor_1^{\ol{R}}(\ol{A},\bb{C}) \\
 \hline 
2 & \ydiagram{1,1,1,1} \oo \ydiagram{2,2} + \ydiagram{2,2} \oo \ydiagram{1,1,1,1} & \ydiagram{4} \oo \ydiagram{2,2} + \ydiagram{2,2} \oo \ydiagram{4} \\[1em]
 \hline 
3 & \ydiagram{3,1,1,1} \oo \ydiagram{2,2,2} + \ydiagram{2,2,2} \oo \ydiagram{3,1,1,1} & \ydiagram{4,1,1} \oo \ydiagram{3,3} + \ydiagram{3,3} \oo \ydiagram{4,1,1} \\[1em]
\end{array}
\]

The rings $A$ and $\ol{A}$ are known as the \defi{special fiber rings} associated with the maps (\ref{eq:phi-wedgephi}) and (\ref{eq:phi-symphi}), and they are natural quotients of the Rees algebras of the ideals $I_2=\langle W\rangle$ of $2\times 2$ minors, and $\ol{I}_2=\langle \ol{W}\rangle$ of $2\times 2$ permanents. The Rees algebras give the bi-graded homogeneous coordinate rings of the graphs of the maps (\ref{eq:phi-wedgephi}) and (\ref{eq:phi-symphi}), and it is an open problem to compute their presentation. We explain a reduction procedure in Section~\ref{sec:Rees}, from which we derive the presentation of the Rees algebra of $I_2$ in the case of $m\times 3$ matrices.

We end the introduction with a summary of the proof strategy for Theorem~\ref{thm:relations-ol-A}.

\medskip

\noindent{\bf Step 1.} Using the results of \cite{BCV} and the equivalence between minors and permanents, we obtain the description of $\Tor_1^{\ol{R}}(\ol{A},\bb{C})_j$ for $j\leq 4$. This step is based on a general duality for polynomial functors, which is explained in Section~\ref{subsec:bi-functors}. The main part of the argument is then concerned with proving the vanishing $\Tor_1^{\ol{R}}(\ol{A},\bb{C})_j = 0$ for $j>4$, and is covered by the following steps.

\medskip

\noindent{\bf Step 2.} We assume that $m\geq n$. We let $S^{(2)}$ denote the second Veronese subring of $S$, which is a finitely generated $\ol{R}$-module (it is infinitely generated over $R$). We compute the zeroth and first Koszul homology of $S$ relative to $\ol{W}$ and restrict to even degrees to find a presentation of $S^{(2)}$ over $\ol{R}$. This presentation has the property that the generators are in degree up to $\lfloor n/2\rfloor$, and the relations in degree up to $\lfloor n/2\rfloor+1$, which implies that
\begin{equation}\label{eq:vanishing-Tor1-S2}
 \Tor_1^{\ol{R}}(S^{(2)},\bb{C})_j = 0 \mbox{ for }j>\lfloor n/2\rfloor+1.
\end{equation}
We note that the calculation of Koszul homology is performed on the minors side (relative to $W$), and then carried over by functoriality to the permanents side (this is a recurring theme in our argument). Since the Koszul homology modules relative to $W$ are equivariant modules supported on the cone over a Segre variety, their structure can be understood using a bivariate version of the Sam--Snowden theory of $\GL$-equivariant modules over polynomial rings in infinitely many variables. We explain the necessary aspects of the theory in Sections~\ref{subsec:filtrations} and~\ref{subsec:filtrations-bigraded}, and perform the Koszul homology calculation in Section~\ref{sec:koszul}.

\medskip

\noindent{\bf Step 3.} The details for this step are discussed in Section~\ref{sec:filtrations-Vero}. We show the existence of a finite filtration
\[ \ol{A} = \ol{M}_0 \subseteq  \ol{M}_1 \subseteq  \ol{M}_2 \subseteq \cdots \subseteq \ol{M}_{\lfloor n/2\rfloor} = S^{(2)}\]
by $\ol{R}$-modules, such that $\ol{M}_r/\ol{M}_{r-1}$ is generated in degree $r$, has degree $(r+1)$ first syzygies, and degree $(r+2)$ second syzygies. As in {\bf Step 2}, we obtain this by translating the description of the beginning of the minimal resolution of $M_r/M_{r-1}$, where 
\[A=M_0\subseteq M_1 \subseteq \cdots \subseteq S^{(2)}\]
is a corresponding (infinite!) filtration by $R$-modules. The shape of the resolution of $M_r/M_{r-1}$ is described using the Kempf--Weyman geometric technique and Bott's theorem. 

It turns out that $\ol{A}$ and $S^{(2)}$ agree in high degree, so $S^{(2)}$ gives a good approximation of $\ol{A}$ for which we can estimate $\Tor_1$ using \eqref{eq:vanishing-Tor1-S2}. We can then think of $\ol{A}$ as being obtained from $S^{(2)}$ by removing finitely many layers, while controlling how $\Tor_1$ changes at each step. We conclude that
\[ \Tor_1^{\ol{R}}(\ol{A},\bb{C})_j = 0 \mbox{ for }j>\lfloor n/2\rfloor+2,\]
and in particular if $n\leq 5$ then the vanishing sought in {\bf Step 1} holds.

\medskip

\noindent{\bf Step 4.} We assume from now on that $n\geq 6$, and prove by induction on the pair $(m,n)$ that Theorem~\ref{thm:relations-ol-A} holds. A useful tool in the study of spaces of tensors, particularly well-suited for inductive proofs, is the subspace variety. After analyzing the equations of a subspace variety $Y$ on the permanents side in Section~\ref{sec:subspace}, we use induction to conclude that for $j\geq 4$, the only non-zero groups $\Tor_1^{\ol{R}}(\ol{A},\bb{C})_j$ may occur when $j\geq n$. Since $n>\lfloor n/2\rfloor+2$, all such groups vanish by {\bf Step 3}.

\medskip

To motivate some of our choices for working on either the minors or permanents side, we note that the passage between the two settings is purely a representation theoretic construction, which behaves poorly relative to the geometry: $\ol{W}$ defines a base-point free linear series, while $W$ has a large base locus, given by a Segre variety; $S^{(2)}$ is the normalization of $\ol{A}$, but it is an infinite $A$-module; the Koszul homology relative to $\ol{W}$ has finite length, while the one with respect to $W$ is built from geometrically interesting modules supported on a Segre cone; the modules $\ol{M}_r/\ol{M}_{r-1}$ have finite length, while each $M_r/M_{r-1}$ arises as a push-forward from a geometric vector bundle. Most of the non-trivial calculations that we make occur on the most geometrically significant side, and are then translated via representation theory to the other side. One exception is the study of the subspace variety, which is equally significant on both sides. The choice we made there was based on the fact that the subspace variety on the permanents side has minimal equations of degree twice as large as that on the minors side, which is crucial in the inductive argument from {\bf Step~4}.

\medskip

\noindent{\bf Organization.} In Section~\ref{sec:prelim} we recall some basic facts about polynomial functors, and give the functorial interpretation of Theorems~\ref{thm:relations-A} and~\ref{thm:relations-ol-A}, explaining how they are equivalent. In Section~\ref{sec:koszul} we compute the first Koszul homology group of the polynomial ring $S$ with respect to the space of $2\times 2$ minors, and derive the corresponding result for permanents. In Section~\ref{sec:filtrations-Vero} we explain a filtration argument that gives an upper bound for the degrees of the minimal generators of $I(\ol{X}_{m,n})$. In Section~\ref{sec:subspace} we find the equations of the subspace variety $Y$. The inductive step in the proof of Theorem~\ref{thm:relations-ol-A} is explained in Section~\ref{sec:induction}. We conclude with a discussion of the defining ideal of the Rees algebra in Section~\ref{sec:Rees}.

\section{Preliminaries}\label{sec:prelim}

The goal of this section is to establish some basic notation concerning partitions and representations of the general linear group, as well as to discuss polynomial functors in the uni- and bivariate setting. We recall the transpose duality for polynomial functors following \cite{sam-snowden}, and explain how the functorial approach gives an equivalence between Theorems~\ref{thm:relations-A} and~\ref{thm:relations-ol-A}.

\subsection{Partitions}\label{subsec:partitions}

We write $\mc{P}$ for the set of all partitions $\ll=(\ll_1\geq\ll_2\geq\cdots\geq 0)$, and write $\mc{P}_n$ for the subset consisting of those $\ll\in\P$ that have at most $n$ parts (that is, $\ll_{n+1}=0$). We write $|\ll|=\ll_1+\ll_2+\cdots$ for the \defi{size} of $\ll$, and write $\ll'$ for the \defi{conjugate partition}, obtained by transposing the corresponding Young diagram. For partitions with repeating parts, we use the abbreviation $(b^a)$ for the sequence $(b,b,\cdots,b)$ of length $a$; for instance $(3,3,3,3,1,1)$ may be written as $(3^4,1^2)$. For $\ll,\mu\in\P$, we write $\mu\geq\ll$ if $\mu_i\geq\ll_i$ for all $i$. We say that $\mu/\ll$ is a \defi{horizontal strip} if $\mu_i\geq\ll_i\geq\mu_{i+1}$ for all $i$, and write $\mu/\ll\in\op{HS}$.

\subsection{Polynomial functors}\label{subsec:functors}

We write $\op{Vec}$ for the category of complex vector spaces. For $d\geq 0$ we consider the tensor power functors $\bb{T}^d:\op{Vec}\lra\op{Vec}$, defined by $\bb{T}^d(V) = V^{\oo d}$. A \defi{polynomial functor} $P:\op{Vec}\lra\op{Vec}$ is a subquotient of a direct sum of tensor power functors. Polynomial functors form a semi-simple abelian category $\mc{V}$, and we refer the reader to \cite[Section~6]{sam-snowden} for its properties. The simple objects are indexed by partitions $\ll$, they are denoted $\bb{S}_{\ll}$, and called \defi{Schur functors}. When $\ll=(d)$, $\SS_{\ll}=\Sym^d$ is the symmetric power functor, while for $\ll=(1^k)$, $\SS_{\ll}=\bw^k$ is the exterior power functor. There is an exact involution $\tau:\mc{V}\lra\mc{V}$, with the property that
\[ \tau(\SS_{\ll}) = \SS_{\ll'}\mbox{ for every partition }\ll\in\P.\]
In particular, $\tau$ interchanges $\Sym^d$ and $\bw^d$. We will be interested in the subcategory $\mc{V}_{gf}$ of \defi{graded-finite} polynomial functors, which are those that decompose as direct sums of Schur functors
\begin{equation}\label{eq:P-in-Vgf}
 P = \bigoplus_{\ll\in\mc{P}} \bb{S}_{\ll}^{\oplus m_{\ll}},\quad m_{\ll}\in\bb{Z}_{\geq 0}.
\end{equation}
Note that $\tau$ preserves $\mc{V}_{gf}$. We write $P_{\ll} = \bb{S}_{\ll}^{\oplus m_{\ll}}$ and refer to it as the \defi{$\ll$-isotypic component} of $P$, and write $P_d$ for the degree $d$ part of $P$, namely
\[ P_d = \bigoplus_{|\ll|=d} \bb{S}_{\ll}^{\oplus m_{\ll}}.\]
A natural pair of elements in $\mc{V}_{gf}$, which are interchanged by $\tau$, is:
\[ \Sym = \bigoplus_{d\geq 0} \Sym^d\quad\mbox{ and }\quad \bw = \bigoplus_{d\geq 0} \bw^d.\]

We let $\GL(V)\simeq\GL_n(\bb{C})$ denote the group of invertible linear transformations of a vector space $V$ of dimension~$n$. For every $P\in\mc{V}$ we have that $P(V)$ is a $\GL(V)$-representation. For Schur functors, we have that $\SS_{\ll}V=0$ when $\ll_1'>n$, that is, when $\ll$ has more than $n$ parts. If $\ll\in\P_n$ then $\SS_{\ll}V$ is an irreducible $\GL$-representation, and moreover, we have that for $\ll,\mu\in\P_n$ there exists an isomorphism $\SS_{\ll}V \simeq \SS_{\mu}V$ as $\GL(V)$-representations if and only if $\ll=\mu$. It follows that for a fixed $\ll\in\P$ we can detect the multiplicity $m_{\ll}$ in (\ref{eq:P-in-Vgf}) by decomposing $P(V)$ into a direct sum of irreducible $\GL(V)$-representations, for any vector space $V$ with $\dim(V)\geq\ll_1'$. This fact will be used repeatedly throughout this article.

\subsection{Bi-variate polynomial functors}\label{subsec:bi-functors}

The category of \defi{bivariate polynomial functors} $P:\op{Vec}\times\op{Vec} \lra \op{Vec}$ is $\mc{V}^{\oo 2}$, with simple objects indexed by pairs $(\ll,\mu)\in\P\times\P$ and denoted $\SS_{\ll}\boxtimes\SS_{\mu}$: we have
\[(\SS_{\ll}\boxtimes\SS_{\mu})(V_1,V_2) = \SS_{\ll}V_1\otimes\SS_{\mu}V_2\mbox{ for every }V_1,V_2\in\op{Vec}.\]
We use the notation $\boxtimes$ to contrast with the univariate functor given by $(\SS_{\ll}\oo\SS_{\mu})(V) = \SS_{\ll}V\oo\SS_{\mu}V$. The subcategory $\mc{V}^{\oo 2}_{gf}$ consists of objects
\[ P = \bigoplus_{\ll,\mu\in\mc{P}} \left(\SS_{\ll}\boxtimes\SS_{\mu}\right)^{\oplus m_{\ll,\mu}},\quad m_{\ll,\mu}\in\bb{Z}_{\geq 0},\]
and we define the $(\ll,\mu)$-isotypic component $P_{\ll,\mu}$ and the bi-graded component $P_{d,e}$ in analogy with the univariate case. The involution $\tau$ induces (commuting) involutions $\tau_1,\tau_2$ on $\mc{V}^{\oo 2}$, which act on the simples by
\[ \tau_1(\SS_{\ll}\boxtimes\SS_{\mu}) = \SS_{\ll'}\boxtimes\SS_{\mu} \qquad \tau_2(\SS_{\ll}\boxtimes\SS_{\mu}) = \SS_{\ll}\boxtimes\SS_{\mu'}.\]
One of the key players in this work is the algebra functor $\mf{S}$ defined by letting 
\[\mf{S}(V_1,V_2) = \Sym(V_1\oo V_2).\]
When $V_1\simeq\bb{C}^m$, $V_2\simeq\bb{C}^n$, we have that $\mf{S}(V_1,V_2)=S$ is the coordinate ring of the space of $m\times n$ matrices. By Cauchy's formula, we have that
\begin{equation}\label{eq:decomp-mfS}
 \mf{S} = \bigoplus_{\ll\in\P} \SS_{\ll} \boxtimes \SS_{\ll},
\end{equation}
and in particular
\[ \tau_1\tau_2\mf{S} = \mf{S}.\]
We define the functor $\mf{E}=\tau_1\mf{S}=\tau_2\mf{S}$, which can be shown to send a pair $(V_1,V_2)$ to the exterior algebra $\bw(V_1\oo V_2)$ (see \cite[Corollary 2.3.3]{weyman}), and which satisfies $\tau_1\tau_2\mf{E} = \mf{E}$. We also consider the (bivariate) functors $\mf{W}$, $\ol{\mf{W}}$, $\mf{R}$, $\ol{\mf{R}}$, defined by
\[\mf{W}(V_1,V_2) = \bw^2 V_1 \oo \bw^2 V_2,\quad \ol{\mf{W}}(V_1,V_2) = \Sym^2 V_1 \oo \Sym^2 V_2,\quad\mf{R} = \Sym \circ\,\mf{W},\quad\ol{\mf{R}} = \Sym \circ\,\ol{\mf{W}}.\]
Since $\mf{W},\ol{\mf{W}}$ are subfunctors of $\mf{S}$, it follows that $\mf{S}$ is naturally an $\mf{R}$- and $\ol{\mf{R}}$-algebra functor. The natural transformations $\phi:\mf{R}\lra\mf{S}$ (resp. $\ol{\phi}:\ol{\mf{R}}\lra\mf{S}$) give rise to subfunctors $\mf{A} = \op{Im}(\phi)$ (resp. $\ol{\mf{A}} = \op{Im}(\ol{\phi})$). With the notation in the introduction, if $S=\mf{S}(V_1,V_2)$ then
\[\mf{R}(V_1,V_2)=R,\quad\ol{\mf{R}}(V_1,V_2)=\ol{R},\quad\mf{A}(V_1,V_2)=A,\quad\ol{\mf{A}}(V_1,V_2)=\ol{A}.\]
Moreover, using the transpose duality functors we get that
\[ \tau_1\tau_2\mf{W} = \ol{\mf{W}},\quad \tau_1\tau_2\mf{R} = \ol{\mf{R}},\quad \tau_1\tau_2\mf{A} = \ol{\mf{A}}.\]
This makes precise the statement that we can exchange minors for permanents in a functorial way.

We write $\mc{P}_{even}$ for the subset of $\mc{P}$ consisting of partitions of even size, and consider the set of partitions 
\[\mc{M}_0 = \{\ll\in\mc{P}_{even}:\ll_1\leq \ll_2+\ll_3+\cdots\}.\]
It follows from \cite[Section~6]{DCEP} (see also \cite[(1.2)]{BCV}) that
\begin{equation}\label{eq:GL-dec-A}
 \mf{A} = \bigoplus_{\ll\in\mc{M}_0}\bb{S}_{\ll}\boxtimes \bb{S}_{\ll},
\end{equation}
and the corresponding formula for $\ol{\mf{A}}$ is obtained by applying $\tau_1\tau_2$.

We define $\Tor_i^{\mf{R}}(\mf{A},\bb{C})$ to be the $i$-th homology of the Koszul complex
\[ \cdots \lra \bw^{i+1}\mf{W} \oo \mf{A} \lra \bw^{i}\mf{W} \oo \mf{A} \lra \bw^{i-1}\mf{W} \oo \mf{A} \lra \cdots\]
and define $\Tor_i^{\ol{\mf{R}}}(\ol{\mf{A}},\bb{C})$ analogously. Since $\tau_1$ and $\tau_2$ are exact, we have that
\begin{equation}\label{eq:tau12-Tor}
 \tau_1\tau_2\Tor_i^{\mf{R}}(\mf{A},\bb{C}) = \Tor_i^{\ol{\mf{R}}}(\ol{\mf{A}},\bb{C})\mbox{ for all }i.
\end{equation}

We note that the functors $\mf{R},\mf{A},\Tor_i^{\mf{R}}(\mf{A},\bb{C})$ etc. are all zero in bi-degree $(d,e)$ unless $d=e=2j$ is even. We will make an abuse of notation and write $\mf{F}_j$ instead of $\mf{F}_{2j,2j}$ when $\mf{F}$ is any of these functors. The discussion above shows that Theorem~\ref{thm:relations-A} is equivalent to the assertion that $\Tor_1^{\mf{R}}(\mf{A},\bb{C})_j=0$ for $j\neq 2,3$ and 
\begin{equation}\label{eq:Tor1-A}
\Tor_1^{\mf{R}}(\mf{A},\bb{C})_2 = \bb{S}_{1,1,1,1}\boxtimes\bb{S}_{2,2} \oplus \bb{S}_{2,2}\boxtimes\bb{S}_{1,1,1,1},\quad\Tor_1^{\mf{R}}(\mf{A},\bb{C})_3 = \bb{S}_{3,1,1,1}\boxtimes\bb{S}_{2,2,2} \oplus \bb{S}_{2,2,2}\boxtimes\bb{S}_{3,1,1,1}.
\end{equation}
Moreover, Theorem~\ref{thm:relations-ol-A} is equivalent to the fact that $\Tor_1^{\ol{\mf{R}}}(\ol{\mf{A}},\bb{C})_j=0$ for $j\neq 2,3$ and 
\begin{equation}\label{eq:Tor1-ol-A}
\Tor_1^{\ol{\mf{R}}}(\ol{\mf{A}},\bb{C})_2 = \bb{S}_{4}\boxtimes\bb{S}_{2,2} \oplus \bb{S}_{2,2}\boxtimes\bb{S}_{4},\quad\Tor_1^{\ol{\mf{R}}}(\ol{\mf{A}},\bb{C})_3 = \bb{S}_{4,1,1}\boxtimes\bb{S}_{3,3} \oplus \bb{S}_{3,3}\boxtimes\bb{S}_{4,1,1}.
\end{equation}
Moreover, we have that (\ref{eq:Tor1-A}) and (\ref{eq:Tor1-ol-A}) are equivalent by (\ref{eq:tau12-Tor}), showing that Theorems~\ref{thm:relations-A} and~\ref{thm:relations-ol-A} are equivalent as well.

\subsection{Highest weight vectors in $S$}\label{subsec:highest-weight}

We let $\GL=\GL(V_1)\times\GL(V_2)$ and decompose $S=\Sym(V_1\oo V_2)$ as a $\GL$-representation (see (\ref{eq:decomp-mfS}))
\begin{equation}\label{eq:decomp-S}
 S = \bigoplus_{\ll\in\mc{P}}\bb{S}_{\ll}V_1 \oo \bb{S}_{\ll}V_2.
\end{equation}
Whenever we need to work with explicit elements of $S$, we assume that $V_1,V_2$ are equipped with fixed bases yielding identifications $V_1\simeq\bb{C}^m$, $V_2\simeq\bb{C}^n$, and $S\simeq\bb{C}[x_{i,j}]$ is a polynomial ring in the entries of a generic $m\times n$ matrix. For every $r\geq 1$ we write
\[\det_r = \det(x_{i,j})_{1\leq i,j\leq r}\]
for the principal $r\times r$ minor coming from the upper left corner of the generic matrix, with the convention that $\det_r=0$ when $r>\min(m,n)$. For a partition $\ll\in\mc{P}$ we let
\[\det_{\ll} = \prod_{i=1}^{\ll_1} \det_{\ll_i'},\]
which is a highest weight vector for the action of $\GL$ on $S$. The component $\bb{S}_{\ll}V_1 \oo \bb{S}_{\ll}V_2$ in (\ref{eq:decomp-S}) is then the $\bb{C}$-linear span of the orbit $\GL\cdot\det_{\ll}$. For instance, when $V_1=V_2=\bb{C}^2$, we have that $\ol{W}=\Sym^2V_1\oo\Sym^2V_2$ is $9$-dimensional, spanned by
\[x_{1,1}^2,\ x_{1,2}^2,\ x_{2,1}^2,\ x_{2,2}^2,\ x_{1,1}x_{1,2},\ x_{1,1}x_{2,1},\ x_{1,2}x_{2,2},\ x_{2,1}x_{2,2},\ x_{1,1}x_{2,2}+x_{1,2}x_{2,1}.\]

\subsection{Filtrations on equivariant modules}\label{subsec:filtrations}

For a vector space $V$ we let $\Sym(V)$ be the corresponding polynomial ring. Throughout this section, by \defi{module} we mean a $\GL(V)$-equivariant $\Sym(V)$-module. The category of such modules is well-understood by \cite{sam-snowden-category}, and we recall here some of the basic facts that will be used later on. By varying $V$, the modules we study give rise to polynomial functors, and we encourage the reader to translate the results here in the language of the earlier sections. For  $\ll \in \P$ define the free module 
\[
F_\lambda(V) = \bb{S}_\ll V \otimes \Sym(V).
\]
Using Pieri's rule \cite[Corollary~2.3.5]{weyman}, we get a multiplicity-free $\GL(V)$-decomposition
\[ F_{\ll}(V) = \bigoplus_{\mu/\ll\in\op{HS}} \bb{S}_{\mu}V.\]
Moreover, it follows from \cite[Proposition~1.3.3]{sam-snowden-category} that if we write $\langle \mc{S}\rangle$ for the submodule generated by a subset $\mc{S}$ of $F_{\ll}(V)$, then
\begin{equation}\label{eq:sub-gen-Smu}
 \langle \bb{S}_{\mu}V\rangle = \bigoplus_{\delta\geq\mu,\delta/\ll\in\op{HS}} \bb{S}_{\delta}V.
\end{equation}
We define $M_\ll(V)$ as the quotient
\[ M_{\ll}(V) = \frac{F_\ll(V)}{\langle \bb{S}_{\mu}V : \mu/\ll\in\op{HS},\ \mu_i>\ll_i\mbox{ for some }i>1\rangle} \overset{(\ref{eq:sub-gen-Smu})}{=} \bigoplus_{d \geq 0} \bb{S}_{\ll_1+d, \ll_2, \ll_3, \ldots}V.\]
We remark that the only submodules of $M_\ll(V)$ have the form $M_{\ll_1 +d, \ll_2, \ll_3, \ldots}(V)$, and they form a chain. Furthermore, there exists a non-zero module map $M_\ll(V) \rightarrow M_\mu(V)$ if and only if $\ll_1 \geq \mu_1$ and $\ll_i = \mu_i$ for all  $i \geq 2$; in this case, the map is injective and unique up to scalar.

We equip $F_\ll(V)$ with a decreasing filtration $\{ \mc{F}^t(F_\ll(V))\}_{t\geq 0}$ by submodules, setting
\begin{equation}\label{eq:filtration-Flam}
\mc{F}^t(F_\ll(V))= \bigoplus_{\substack{\mu/\lambda\in\op{HS} \\ t\leq \mu_2 + \mu_3 + \cdots }} \bb{S}_\mu V.
\end{equation}
We denote the graded components of associated graded modules by $\gr^t(-)$.
Note that for $\ll=(0,0,\ldots)$ we have $F_{\ll}(V)=\Sym(V)$, and $\mc{F}^0(\Sym(V)) = \Sym(V) $, $\mc{F}^1(\Sym(V)) = 0$, so $\gr(\Sym(V)) =\gr^0(\Sym(V)) = \Sym(V)$. 
It follows that, for any $\ll$, the associated graded module $\gr(F_\ll(V))$ is a module over $\Sym(V)$ (and $\GL(V)$-equivariant), and moreover, we have a module isomorphism
\begin{equation}\label{eq:grt-Flam}
\gr^t(F_\ll (V)) = \frac{\mc{F}^t(F_\ll(V))}{\mc{F}^{t+1}(F_\ll(V))} = \bigoplus_{\substack{\mu/\ll \in \op{HS} \\ \mu_1 = \ll_1 \\ t = \mu_2 + \mu_3 +\cdots }} M_\mu(V).
\end{equation}
We extend this filtration to  direct sums of $F_\lambda(V)$'s, and note that by (\ref{eq:filtration-Flam}) every $\GL(V)$-equivariant map automatically respects the filtration. We note also that taking $\gr(\cdot)$ only affects the $\Sym(V)$-module structure, but not the $\GL$-structure.
In other words, if $\varphi: F_2 \rightarrow F_1$ is an equivariant map of  finite free modules, then for the associated graded map
$\gr(\varphi) : \gr(F_2) \rightarrow \gr(F_1)$
we have  $\varphi = \gr(\varphi)$ as $\GL(V)$-linear maps (but not as $\Sym(V)$-linear maps). In the special case of the (unique up to scaling) map $\varphi:F_{\ll}(V) \lra F_{\gamma}(V)$, where $\ll/\gamma\in\op{HS}$, we have that the induced map $\gr(\varphi)$ embeds
\[ M_{\mu}(V) \hookrightarrow M_{\delta}(V)\]
if there exists $\delta$ with $\delta/\gamma\in\op{HS}$, $\delta_1=\gamma_1$ and $\delta_i=\mu_i$ for $i\geq 2$, and it sends $M_{\mu}(V)$ to zero if no such $\delta$ exists.

\begin{example}\label{ex:gr-map}
 If $\ll=(3,1)$ and $\gamma=(1,1)$ then the non-zero components of $\gr(F_{\ll}(V))$ and $\gr(F_{\gamma}(V))$ are:
\[
\setlength{\extrarowheight}{3pt}
\ytableausetup{smalltableaux,aligntableaux=center}
\renewcommand{\arraystretch}{1.5}
\begin{array}{c|c|c|c|c}
 t & 1 & 2 & 3 & 4 \\
 \hline 
\gr^t(F_{\ll}(V)) & M_{\ydiagram{3,1}} & M_{\ydiagram{3,2}} \oplus M_{\ydiagram{3,1,1}} & M_{\ydiagram{3,3}} \oplus M_{\ydiagram{3,2,1}} & M_{\ydiagram{3,3,1}} \\[2em]
 \hline 
\gr^t(F_{\gamma}(V)) & M_{\ydiagram{1,1}} & M_{\ydiagram{1,1,1}}  & & \\
\end{array}
\]
Let $\varphi:F_{\ll}(V) \lra F_{\gamma}(V)$ be as above.
The only non-zero part of $\gr(\varphi)$ is given by the inclusion of $M_{3,1}$ into $M_{1,1}$, and that of $M_{3,1,1}$ into $M_{1,1,1}$.
\end{example}

\subsection{Filtrations in the bi-graded setting}\label{subsec:filtrations-bigraded}

We consider now a pair of vector spaces $V_1,V_2$, and the associated polynomial ring $\Sym(V_1) \otimes \Sym(V_2) = \Sym(V_1\oplus V_2)$. We write $\GL=\GL(V_1)\times\GL(V_2)$, and call \defi{module} a $\GL$-equivariant $\Sym(V_1) \otimes \Sym(V_2)$-module. We have in particular for $\ll,\mu\in \P$ the free module
\[ F_\ll(V_1) \otimes F_\mu(V_2),\]
which is equipped with a bi-filtration by submodules, given by
\[
\mc{F}^{s,t}(F_\ll(V_1) \otimes F_\mu(V_2))= \mc{F}^s(F_\ll(V_1)) \otimes \mc{F}^t(F_\mu(V_2))\mbox{ for }s,t\in \bb{N}.
\]
The associated graded components are 
\[
\gr^{(s,t)} \big(F_\ll(V_1) \otimes F_\mu(V_2)\big)= \gr^s(F_\ll (V_1)) \oo \gr^t(F_\ll (V_2)).
\] 
As before, all module maps respect this filtration (since they are assumed to be $\GL$-equivariant).

All our modules are naturally bi-graded by placing each irreducible $\GL$-representation $\bb{S}_\delta V_1 \oo \bb{S}_\gamma V_2$ in bi-degree $(|\delta|,|\gamma|)$. We have a \defi{diagonal} functor $\Delta$ that picks up the symmetric bi-degrees in each module:
\[ \Delta(M) = \bigoplus_{d\geq 0} M_{(d,d)}.\]
If we apply $\Delta$ to the ring itself we obtain the coordinate ring of the Segre product $\bb{P}V_1\times\bb{P}V_2$:
\[ \Delta\bigl(\Sym(V_1) \otimes \Sym(V_2) \bigr) = \bigoplus_{d\geq 0} \Sym^d V_1 \oo \Sym^d V_2 = S/I_2,\]
where $S$ and $I_2$ are as defined in the Introduction. We assume that $|\ll|=|\mu|$ and obtain $S/I_2$-modules
\begin{eqnarray}
F_{\ll,\mu}  &=& \Delta \big(F_\ll(V_1) \otimes F_\mu(V_2)\big)= \bb{S}_\ll V_1 \otimes \bb{S}_\mu V_2 \otimes S/I_2,
\\
M_{\ll,\mu} &=& \Delta\big( M_\ll(V_1) \otimes M_\mu(V_2)\big)
 = \bigoplus_{d\geq 0} \bb{S}_{\ll_1+d, \ll_2, \ll_3, \ldots} V_1 \otimes \bb{S}_{\mu_1+d, \mu_2, \mu_3, \ldots} V_2 \label{charM}.
 \end{eqnarray}

 Applying the diagonal functor we obtain an induced bi-filtration on $F_{\ll,\mu}$, with associated graded components
\[
\gr^{(s,t)}(F_{\ll,\mu}) = \Delta\big(\gr^s(F_\ll (V_1)) \oo \gr^t(F_\mu (V_2))\big) = \bigoplus_{(\tau,\theta)\in \mc{N}^{s,t}_{\ll,\mu}} \bb{S}_\tau V_1 \oo \bb{S}_\theta V_2,
\]
where
\[
\mc{N}^{s,t}_{\ll,\mu} = \Big \{ (\tau,\theta)\in \P \times \P \, : \, \tau/\lambda, \, \theta/\mu\in \op{HS},\, s = \tau_2 + \tau_3 +\cdots,\, t = \theta_2+ \theta_3 + \cdots, \text{ and } |\tau| = |\theta|\Big\}.
\]
In other words,  we have a direct sum module decomposition
\[
\gr(F_{\ll,\mu}) = \bigoplus_{(\a,\b)\in \mc{N}_{\ll,\mu}} M_{\a,\b},
\]
where
\begin{equation}\label{eq:index-set-associated-graded}
\mc{N}_{\ll,\mu} = \Big \{ (\alpha, \beta)\in \P \times \P \, : \, \alpha/\lambda, \, \beta/\mu\in \op{HS} \text{  and  }
\left( \alpha_1 = \ll_1 \text{ or } \beta_1 = \mu_1\right)\Big\}.
\end{equation}

\begin{remark}\label{rem:basic-facts}
As for the case of one vector space $V$, the following facts hold:
\begin{enumerate}
\item The only submodules of $M_{\ll,\mu}$ are of the form $M_{(\ll_1 +d, \ll_2, \ll_3, \ldots),(\mu_1 +d, \mu_2, \mu_3, \ldots)}$, and they form a chain.
\item There exists a non-zero $S/I_2$-linear map  $M_{\ll,\mu} \rightarrow M_{\alpha, \beta}$ if and only if $\ll_1 \geq \alpha_1, \mu_1 \geq \beta_1, \ll_i = \alpha_i,  \mu_i=\beta_i$ for  $i \geq 2$;
 in this case, the map is injective and unique up to scalar.
\item If $\varphi: F_2 \rightarrow F_1$ is an equivariant map of  finite free $S/I_2$-modules, 
then for the associated graded map
$\gr(\varphi) : \gr(F_2) \rightarrow \gr(F_1)$
we have  $\varphi = \gr(\varphi)$ as $\GL$-equivariant maps (but not as maps of $S/I_2$-modules).
\item Specializing (3) to the case when $F_2=F_{\ll,\mu}$ and $F_1=F_{\a,\b}$, we get that $\gr(\varphi)$ includes each component $M_{\tau,\theta}$ into a corresponding $M_{\gamma,\delta}$ whenever possible.
More precisely,
we have $\gr(\varphi) (M_{\tau,\theta})\subseteq M_{\gamma,\delta}$ if 
\[(\tau,\theta)\in\mc{N}_{\ll,\mu},\quad (\gamma,\delta)\in\mc{N}_{\a,\b},\quad  \tau_1 \geq \gamma_1,\ \theta_1 \geq \delta_1,\quad \tau_i=\gamma_i,\ \theta_i=\delta_i\mbox{ for }i\geq 2,\]
and $\gr(\varphi) (M_{\tau,\theta})=0$ if no such $\gamma, \delta$ exist.
\end{enumerate}
\end{remark}

Explicit examples of the bi-filtrations discussed above, and the corresponding induced maps, will appear in the proof of Theorem~\ref{thm:H1K}.

\section{The first Koszul homology groups for $2\times 2$ minors}\label{sec:koszul}

The goal of this section is to describe the zeroth and first Koszul homology groups associated with the space $W\subset S$ spanned by the $2\times 2$ minors. We write $\mc{K}_{\bullet}(W)$ for the Koszul complex whose $i$-th term is
\[ \mc{K}_i(W) = \bw^i W \oo S,\]
and note that
\[ H_i(\mc{K}_{\bullet}(W)) = \Tor_i^R(S,\bb{C}).\]
Moreover, since $I_2= \langle W \rangle \subseteq S$ then 
\begin{equation}\label{eq:H0KW}
 H_0(\mc{K}_{\bullet}(W)) = S/I_2 = \bigoplus_{d\geq 0}\Sym^d V_1 \oo \Sym^d V_2
\end{equation}
is the homogeneous coordinate ring of the Segre product $\bb{P}V_1 \times \bb{P}V_2$. We prove the following.

\begin{theorem}\label{thm:H1K}
 Let $K=H_1(\mc{K}_{\bullet}(W))$. The graded components of $K$ are described as $\GL$-representations by $K_d = 0$ for $d<3$,
 \[ K_3=\bb{S}_{2,1}V_1\oo \bb{S}_{1,1,1} V_2 \oplus \bb{S}_{1,1,1} V_1 \oo \bb{S}_{2,1}V_2,\]
 \[ K_4=\bb{S}_{2,2}V_1\oo \bb{S}_{1,1,1,1} V_2 \oplus \bb{S}_{1,1,1,1} V_1 \oo \bb{S}_{2,2}V_2 \oplus \bb{S}_{2,1,1} V_1 \oo \bb{S}_{2,1,1}V_2 \oplus \bb{S}_{2,1,1} V_1 \oo \bb{S}_{3,1}V_2 \oplus \bb{S}_{3,1} V_1 \oo \bb{S}_{2,1,1}V_2,\]
 \[ K_{d} = \bb{S}_{d-2,1,1}V_1 \oo \bb{S}_{d-2,1,1}V_2 \oplus \bb{S}_{d-2,1,1} V_1 \oo \bb{S}_{d-1,1}V_2 \oplus \bb{S}_{d-1,1} V_1 \oo \bb{S}_{d-2,1,1}V_2,\mbox{ for }d\geq 5.\]
\end{theorem}

Recall from the Introduction that $\ol{W}=\Sym^2 V_1 \oo \Sym^2 V_2$ is the space spanned by the $2\times 2$ permanents, and $\ol{I}_2=\langle \ol{W} \rangle\subset S$, then it follows by transpose duality from (\ref{eq:H0KW}) that
\[ H_0(\mc{K}_{\bullet}(\ol{W})) = S/\ol{I}_2 = \bigoplus_{d\geq 0} \bw^d V_1 \oo \bw^d V_2 = \bigoplus_{d=0}^n \bw^d V_1 \oo \bw^d V_2,\]
where the last equality follows since $\bw^d V_2=0$ for $d>n=\dim(V_2)$. Moreover, from Theorem~\ref{thm:H1K} we get:

\begin{theorem}\label{thm:H1K-bar}
 Let $\ol{K} = H_1(\mc{K}_{\bullet}(\ol{W}))$. We have $\ol{K}_d=0$ for $d<3$, and
\[ \ol{K}_3=\bb{S}_{2,1}V_1\oo \bb{S}_{3} V_2 \oplus \bb{S}_{3} V_1 \oo \bb{S}_{2,1}V_2,\]
 \[ \ol{K}_4=\bb{S}_{2,2}V_1\oo \bb{S}_{4} V_2 \oplus \bb{S}_{4} V_1 \oo \bb{S}_{2,2}V_2 \oplus \bb{S}_{3,1} V_1 \oo \bb{S}_{3,1}V_2 \oplus \bb{S}_{3,1} V_1 \oo \bb{S}_{2,1,1}V_2 \oplus \bb{S}_{2,1,1} V_1 \oo \bb{S}_{3,1}V_2,\]
 \[ \ol{K}_{d} = \bb{S}_{3,1^{d-3}}V_1 \oo \bb{S}_{3,1^{d-3}}V_2 \oplus \bb{S}_{3,1^{d-3}} V_1 \oo \bb{S}_{2,1^{d-2}}V_2 \oplus \bb{S}_{2,1^{d-2}} V_1 \oo \bb{S}_{3,1^{d-3}}V_2,\mbox{ for }d\geq 5.\]
 In particular, since $\bb{S}_{\ll}V_2=0$ when $\ll$ has more than $n$ parts, it follows that $\ol{K}_d=0$ for $d\geq n+3$. 
\end{theorem}

\begin{proof}[Proof of Theorem~\ref{thm:H1K}]
 We write $\pd_i$ for the $i$-th differential in $\mc{K}_{\bullet}(W)$, so that $K = \ker(\pd_1) / \op{Im}(\pd_2)$.  
 We recall the beginning of the  Lascoux minimal free resolution of $S/I_2$ (see  \cite[Section 6.1]{weyman})
  \[ W^3 \oo S \overset{\delta_4}{\lra} W^2 \oo S \overset{\delta_3}{\lra} W^1 \oo S \overset{\delta_2}{\lra} W\oo S \overset{\delta_1}{\lra} S\]
where
\begin{eqnarray*}
W^1 &=& \bb{S}_{1,1,1} V_1 \otimes \bb{S}_{2,1} V_2 
\, \oplus \,
\bb{S}_{2,1} V_1 \otimes \bb{S}_{1,1,1} V_2,
\\
W^2 & = & \bb{S}_{2,1,1} V_1 \otimes \bb{S}_{2,1,1} V_2 
\, \oplus \,
\bb{S}_{1,1,1,1} V_1 \otimes \bb{S}_{3,1} V_2 
\, \oplus \,
\bb{S}_{3,1} V_1 \otimes \bb{S}_{1,1,1,1} V_2, 
 \\
W^3 & = & 
\bb{S}_{3,1,1} V_1 \otimes \bb{S}_{2,1,1,1} V_2 
\, \oplus \,
\bb{S}_{2,1,1,1} V_1 \otimes \bb{S}_{3,1,1} V_2 
\, \oplus \,
\bb{S}_{2,2,2} V_1 \otimes \bb{S}_{2,2,2} V_2.
\end{eqnarray*}
We will also need the decomposition
\[
\bigwedge^2 W  = 
\bb{S}_{2,1,1} V_1 \otimes \bb{S}_{2,2} V_2 
\, \oplus \,
\bb{S}_{2,1,1} V_1 \otimes \bb{S}_{1,1,1,1} V_2 
  \oplus 
\bb{S}_{2,2} V_1 \otimes \bb{S}_{2,1,1} V_2 
\, \oplus \,
\bb{S}_{1,1,1,1} V_1 \otimes \bb{S}_{2,1,1} V_2.
\]

Since $\delta_1 = \pd_1$ we have that $K = \op{Im} (\delta_2) / \op{Im} (\pd_2).$
Factoring  $\pd_2 : \bigwedge^2 W \otimes S \rightarrow \ker(\pd_1) \subseteq W\otimes S$ 
through   $\delta_2 : W^1 \otimes S \twoheadrightarrow \ker(\delta_1)$
we obtain a  map ${\epsilon_2} : \bigwedge^2 W \otimes S \rightarrow W^1 \otimes S$ and a  linear  presentation 
$$
\left(W^2 \oplus\bigwedge^2 W\right) \otimes S \xrightarrow{\delta_3 \oplus {\epsilon_{2}}} W^1 \otimes S \rightarrow K \rightarrow 0.
$$
Tensoring with $S/I_2$ over $S$ we obtain the presentation 
\begin{equation}\label{eq:SI2-pres-K}
F_2 = \left(W^2 \oplus\bigwedge^2 W\right) \otimes S/I_2 \xrightarrow{\,\varphi \,= \,\overline{\delta_3} \oplus \ol{\epsilon_{2}}\,} F_1 = W^1 \otimes S/I_2 \rightarrow K \rightarrow 0
\end{equation}
where we used that $K$ is annihilated by $I_2$ (see \cite[Prop.~17.14]{eisenbud-CA}).
 
In order to find the $\GL$-decomposition of $K$ we determine the image of $\varphi$.
We consider the filtrations of $F_1, F_2$ constructed in Section \ref{subsec:filtrations-bigraded}.
By Remark \ref{rem:basic-facts} (3) it suffices to study the associated graded map $\gr(\varphi) : \gr(F_2) \rightarrow \gr(F_1)$.
The components of $\gr(F_1) = \gr(F_{(1,1,1),(2,1)} \oplus F_{(2,1),(1,1,1)})$ are determined by 
 \eqref{eq:index-set-associated-graded}:
for the summand $\gr(F_{(1,1,1),(2,1)})$ we have
\begin{align*}
&M_{(1,1,1),(2,1)}&
&M_{(1,1,1,1),(2,1,1)}&
&M_{(1,1,1,1),(2,2)}&
&M_{(1,1,1,1),(3,1)}&
\\
&M_{(2,1,1),(2,1,1)}&
&M_{(2,1,1),(2,2)}&
&M_{(2,1,1,1),(2,2,1)}&
&M_{(3,1,1),(2,2,1)}&
\end{align*}
and, symmetrically,   for the  summand $ \gr(F_{(2,1),(1,1,1)})$ we have
\begin{align*}
&M_{(2,1),(1,1,1)}&
&M_{(2,1,1),(1,1,1,1)}&
&M_{(2,2),(1,1,1,1)}&
&M_{(3,1),(1,1,1,1)}&
\\
&M_{(2,1,1),(2,1,1)}&
&M_{(2,2),(2,1,1)}&
&M_{(2,2,1),(2,1,1,1)}&
&M_{(2,2,1),(3,1,1)}.&
\end{align*}
We are going to calculate $M_{\lambda, \mu} \cap \op{Im}(\gr(\varphi)) $
for each of the 16 terms $M_{\lambda, \mu}$ above. 
Observe that each $M_{\lambda, \mu}$ appears exactly once, with the exception of $M_{(2,1,1),(2,1,1)}$,
therefore,
by symmetry, it suffices to deal with the first 8 cases and with the second copy $M_{(2,1,1),(2,1,1)}$ appearing in the last 8.
\begin{enumerate}
\item\label{ItemDisjoint}  $ M_{(1,1,1),(2,1)} \cap \op{Im}(\gr(\varphi)) = 0 $.
\\
It follows from Remark \ref{rem:basic-facts} (2) by inspecting the graded components of $\gr(F_2)$.
\item\label{ItemKoszul} $ M_{(1,1,1,1),(2,1,1)} \subseteq \op{Im}(\gr(\varphi)) $.
\\
 $\bb{S}_{1,1,1,1}V_1 \otimes \bb{S}_{2,1,1}V_2$ is part of the minimal generating set of $\bigwedge^2 W \otimes S/I_2$.
Its image is  non-zero under $\pd_2$, therefore also under the lift $\epsilon_2$, 
and by degree reasons also under $\ol{\epsilon_2}$.
Thus $M_{(1,1,1,1),(2,1,1)}\subseteq \op{Im}(\gr(\varphi)) $ by Remark \ref{rem:basic-facts} (2).

\item\label{ItemSyzygy} $ M_{(1,1,1,1),(2,2)} \cap \op{Im}(\gr(\varphi)) =  M_{(2,1,1,1),(3,2)}  $.
\\
$\bb{S}_{1,1,1,1}V_1\otimes \bb{S}_{2,2}V_2$ does not appear in $F_2$. 
On the other hand, 
$\bb{S}_{2,1,1,1}V_1\otimes \bb{S}_{3,2}V_2$
appears in $W^2\otimes S$, and its image in $W^1\otimes S$ under  $\delta_3$ is non-zero because $\bb{S}_{2,1,1,1}V_1\otimes \bb{S}_{3,2}V_2$ does not appear in $W^3\otimes S$.
Note that $\bb{S}_{2,1,1,1}V_1\otimes \bb{S}_{3,2}V_2$ appears  once in $W^1\otimes S$, 
hence it avoids the kernel of $W^1\otimes S \twoheadrightarrow W^1\otimes S/I_2$.
We deduce  that the image of $\bb{S}_{2,1,1,1}V_1\otimes \bb{S}_{3,2}V_2$ under $\ol{\delta_3}$ is non-zero,
and
the desired conclusion follows by Remark \ref{rem:basic-facts} (2).

\item\label{ItemLascoux} $ M_{(1,1,1,1),(3,1)} \subseteq \op{Im}(\gr(\varphi)) $.
\\
$\bb{S}_{1,1,1,1}V_1 \otimes \bb{S}_{3,1}V_2$ is part of the minimal generating set of $ W^2 \otimes S/I_2$.
Its image is  non-zero under  $\delta_3$, 
and by degree reasons also under $\ol{\delta_3}$.
Thus $ M_{(1,1,1,1),(3,1)} \subseteq \op{Im}(\gr(\varphi)) $  by Remark \ref{rem:basic-facts} (2).
\item  $M_{(2,1,1),(2,1,1)}^{\oplus 2} \cap \op{Im}(\gr(\varphi)) \cong M_{(2,1,1),(2,1,1)}$.
\\
By Remark \ref{rem:basic-facts} (2)
the only $M_{\ll,\mu}\subseteq \gr(F_1)$ that can map to  $M_{(2,1,1),(2,1,1)}$
comes from the subspace $\bb{S}_{2,1,1}V_1 \otimes \bb{S}_{2,1,1}V_2$ that is part of the minimal generating set of $W^2\otimes S/I_2$.
As in item \eqref{ItemLascoux} we conclude that $\op{Im}(\gr(\varphi))$ contains (exactly) one copy of $M_{(2,1,1),(2,1,1)}$.
\item $ M_{(2,1,1),(2,2)} \subseteq \op{Im}(\gr(\varphi)) $ as in item \eqref{ItemKoszul}.
\item $ M_{(2,1,1,1),(2,2,1)} \subseteq \op{Im}(\gr(\varphi)) $ as in the second part of  item \eqref{ItemSyzygy}.
\item $ M_{(3,1,1),(2,2,1)} \subseteq \op{Im}(\gr(\varphi)) $ as in the second part of  item \eqref{ItemSyzygy}.
\end{enumerate}

We conclude that the associated graded module of the induced filtration on $K$ is
\[
\gr(K) =  M_{(1,1,1),(2,1)} \oplus  M_{(2,1),(1,1,1)} \oplus M_{(2,1,1),(2,1,1)} \oplus \bb{S}_{1,1,1,1}V_1\otimes \bb{S}_{2,2}V_2 \oplus \bb{S}_{2,2}V_1\otimes \bb{S}_{1,1,1,1}V_2.
\]
By (\ref{charM}), this implies  the $\GL$-decomposition stated in the theorem.
\end{proof}

\section{Filtrations on the second Veronese subring}\label{sec:filtrations-Vero}

We let $S^{(2)}$ denote the second Veronese subring of $S$, with the grading given by
\[ S^{(2)}_d = \Sym^{2d}(V_1\oo V_2) = \bigoplus_{\ll\vdash 2d}\bb{S}_{\ll}V_1 \oo \bb{S}_{\ll}V_2.\]
We can think of $S^{(2)}$ as a graded module over the (standard) graded polynomial ring $R$, and it follows from Section~\ref{sec:koszul} that
\begin{equation}\label{eq:Tor0-R-S2}
\Tor_0^R(S^{(2)},\bb{C})_d = \Sym^{2d}V_1 \oo \Sym^{2d}V_2,
\end{equation}
and (see Theorem~\ref{thm:H1K} for the description of $K$)
\[ \Tor_1^R(S^{(2)},\bb{C})_d = K_{2d}.\]
The $R$-submodule of $S^{(2)}$ generated by $1$ is $A$, so it is clear that $\Tor_0^R(A,\bb{C}) = \bb{C}$ (concentrated in degree $0$). 

The formula (\ref{eq:Tor0-R-S2}), which describes the minimal generators of $S^{(2)}$ as an $R$-module, provides a natural increasing filtration of $S^{(2)}$ by $R$-submodules
\begin{equation}\label{eq:filtration-M}
 A = M_0 \subseteq M_1 \subseteq M_2 \subseteq \cdots \subseteq S^{(2)},
\end{equation}
where $M_i$ is the $R$-submodule generated by $\bigoplus_{d=0}^i \Sym^{2d}V_1 \oo \Sym^{2d}V_2$.

The goal of this section is to prove the following:

\begin{theorem}\label{thm:Tor-grM}
 For each $r>0$ and for $i=0,1,2$ we have that
 \[ \Tor_i^R(M_r/M_{r-1},\bb{C})_j = 0\mbox{ if }j\neq i+r.\]
\end{theorem}

Applying transpose duality, it follows from (\ref{eq:filtration-M}) that we have a filtration
\begin{equation}\label{eq:filtration-Mbar}
 \ol{A} = \ol{M}_0 \subseteq  \ol{M}_1 \subseteq  \ol{M}_2 \subseteq \cdots \subseteq S^{(2)},
\end{equation}
where $\ol{M}_i$ is the $\ol{R}$-submodule generated by $\bigoplus_{d=0}^i \bw^{2d}V_1 \oo \bw^{2d}V_2$. Unlike (\ref{eq:filtration-M}), the filtration (\ref{eq:filtration-Mbar}) is finite, and we have $\ol{M}_{\lfloor n/2\rfloor}=S^{(2)}$. It follows from Theorem~\ref{thm:Tor-grM} that for $i=0,1,2$ we have
 \begin{equation}\label{eq:TorMbar-vanish}
  \Tor_i^{\ol{R}}(\ol{M}_r/\ol{M}_{r-1},\bb{C})_j = 0\mbox{ if }j\neq i+r.
 \end{equation}

\begin{corollary}\label{cor:Tor1-low-degree}
 We have that $\Tor_1^{\ol{R}}(\ol{A},\bb{C})_j=0$ for $j>\lfloor n/2\rfloor+2$.
\end{corollary}

\begin{proof} We prove by descending induction on $r$ that $\Tor_1^{\ol{R}}(\ol{M}_r,\bb{C})_j=0$ for $j>\lfloor n/2\rfloor+2$. The base case is $r=\lfloor n/2\rfloor$, which follows from Theorem~\ref{thm:H1K-bar}, since
\[ \Tor_1^{\ol{R}}(S^{(2)},\bb{C})_j = \ol{K}_{2j}=0\mbox{ for }2j\geq n+3.\]
For the induction step, we have an exact sequence
\[\cdots \lra \Tor_2^{\ol{R}}(\ol{M}_r/\ol{M}_{r-1},\bb{C}) \lra \Tor_1^{\ol{R}}(\ol{M}_{r-1},\bb{C}) \lra \Tor_1^{\ol{R}}(\ol{M}_r,\bb{C}) \lra \cdots\]
and the desired conclusion follows using (\ref{eq:TorMbar-vanish}) with $i=2$, and the induction hypothesis.
\end{proof}

We now proceed with the proof of Theorem \ref{thm:Tor-grM}. The module $M_r/M_{r-1}$ has space of minimal generators given by $\Sym^{2r}V_1\oo\Sym^{2r}V_2$. In higher degrees, the $\GL$-equivariant structure is given as follows (see also \cite[Remark~1.9]{BCV}).

\begin{lemma}\label{lem:dec-Mr-Mr-1}
The $\GL$-decomposition of $M_r/M_{r-1}$ is given by
\begin{equation}\label{eq:decomp-Mr-Mr-1}
\bigoplus_{\ll\in\mc{M}_r\setminus\mc{M}_{r-1}} \bb{S}_{\ll}V_1 \oo \bb{S}_{\ll}V_2,
\end{equation}
where $\mc{M}_r = \{ \ll\in\mc{P}_{even} : 2\ll_1 - |\ll| \leq 2r\}$.
\end{lemma}

\begin{proof} To prove this statement, we need to check that
\[M_r = \bigoplus_{\ll\in\mc{M}_r} \bb{S}_{\ll}V_1 \oo \bb{S}_{\ll}V_2.\]
The case $r=0$ follows from (\ref{eq:GL-dec-A}). Since $M_r/M_{r-1}$ is generated by $\Sym^{2r}V_1\oo\Sym^{2r}V_2$, it is a quotient of
\[\Sym^{2r}V_1\oo\Sym^{2r}V_2 \oo A.\]
Using Pieri's rule, if $\ll\in\mc{M}_0$ then
\[\bb{S}_{\ll}V \oo \Sym^{2r}V = \bigoplus_{\mu} \bb{S}_{\mu}V,\]
where $\mu$ varies over partitions in $\mc{M}_r$. It follows that $M_r/M_{r-1}$ is a subrepresentation of (\ref{eq:decomp-Mr-Mr-1}). To prove the reverse inclusion, it suffices to check that if $\ll\in\mc{M}_r$ then a highest weigh vector in $\bb{S}_{\ll}V_1\oo\bb{S}_{\ll}V_2$ belongs to $M_r$ (see Section~\ref{subsec:highest-weight} for the notation). Note that if $\ll\in\mc{M}_r\setminus\mc{M}_{r-1}$ then $\ll_1-2r\geq\ll_2$, so $\mu=(\ll_1-2r,\ll_2,\ll_3,\cdots)$ satisfies $\mu\in\mc{M}_0$, and the corresponding highest weight vectors satisfy
\[ \det_{\ll} = x_{1,1}^{2r}\cdot \det_{\mu}.\]
Since $\det_{\mu}\in A=M_0$ and $x_{1,1}^{2r}\in\Sym^{2r}V_1\oo\Sym^{2r}V_2$, it follows that $\det_{\ll}\in M_r$.
\end{proof}

We are going to realize each $M_r/M_{r-1}$ as the global sections of a vector bundle, and compute its syzygies using the Kempf--Weyman geometric technique.
 We let $X=\bb{P}V_1 \times \bb{P}V_2$, where $\bb{P}V_i$ is the projective space of one dimensional quotients of $V_i$. We consider the tautological exact sequence on $\bb{P}V_i$,
\begin{equation}\label{eq:tautological-ses}
 0\lra \mc{R}_i \lra V_i \oo \mc{O}_{\bb{P}V_i} \lra \mc{Q}_i \lra 0,
\end{equation}
where $\mc{Q}_i$ is the tautological quotient line bundle on $\bb{P}V_i$ (often denoted $\mc{O}_{\bb{P}V_i}(1)$). Using \cite[Exercise~II.5.16]{hartshorne}, we get exact sequences
\begin{equation}\label{eq:ses-wedge2}
 0\lra \bw^2\mc{R}_i \lra \bw^2V_i \oo \mc{O}_{\bb{P}V_i} \lra \mc{R}_i \oo \mc{Q}_i \lra 0.
\end{equation}
We write $\pi_i:X \lra \bb{P}V_i$ for the natural projection maps, and define $\eta = \pi_1^*(\mc{R}_1 \oo \mc{Q}_1) \oo \pi_2^*(\mc{R}_2 \oo \mc{Q}_2)$, $\mc{L} = \pi_1^*(\mc{Q}_1) \oo \pi_2^*(\mc{Q}_2)$.
Pulling back to $X$ the two sequences in (\ref{eq:ses-wedge2}) and tensoring them together, we get an exact sequence
\[ 0 \lra \xi \lra W \oo \mc{O}_X \lra \eta \lra 0,\]
where $\xi$ is a locally free sheaf and can also be obtained as an extension
\[0 \lra \xi_1 \lra \xi \lra \xi_2 \lra 0\]
where
\[\xi_1 = \bw^2 V_1 \oo \pi_2^*\left(\bw^2\mc{R}_2\right)\mbox{ and }\xi_2=\pi_1^*\left(\bw^2\mc{R}_1\right) \oo \pi_2^*\left(\mc{R}_2\oo \mc{Q}_2\right).\]
We will need the following consequence of Bott's theorem \cite[Corollary 4.1.9]{weyman}.

\begin{lemma}\label{lem:vanishing-xi1-xi2}
 Suppose that $u,v\in\bb{Z}_{\geq 0}$, $j,r\in\bb{Z}_{\geq 1}$, and $u+v\leq j+2$. We have
 \[ H^j\left(X,\mc{L}^{2r}\oo\bw^u\xi_1 \oo \bw^v\xi_2\right) = 0.\]
\end{lemma}

\begin{proof}
 Using Cauchy's formula, we have that
 \[\bw^u\xi_1 = \bigoplus_{\a\vdash u}\bb{S}_{\a}\left(\bb{S}_{1,1}V_1\right) \oo \pi_2^*\left(\bb{S}_{\a'}\left(\bb{S}_{1,1}\mc{R}_2\right)\right),\quad \bw^v\xi_2 = \bigoplus_{\b\vdash v}\pi_1^*\left(\bb{S}_{\b}\left(\bb{S}_{1,1}\mc{R}_1\right)\right) \oo \pi_2^*\left(\bb{S}_{\b'}\left(\mc{R}_2\oo \mc{Q}_2\right)\right).\]

 Using the projection formula and K\"unneth's formula, it suffices to check that if the following hold
 \begin{itemize}
 \item $\ll\vdash 2v$ is such that $\bb{S}_{\ll}$ appears in the plethysm $\bb{S}_{\b}\circ\bb{S}_{1,1}$ for some $\b\vdash v$,
 \item $\mu\vdash (2u+v)$ is such that $\bb{S}_{\mu}$ appears in the tensor product $(\bb{S}_{\a'}\circ\bb{S}_{1,1})\oo \bb{S}_{\b'}$ for $\a\vdash u$, $\b\vdash v$,
 \item $a,b\in\bb{Z}_{\geq 0}$ satisfy $a+b=j$,
 \end{itemize}
 then $H^a(\bb{P}V_1,\mc{Q}_1^{2r}\oo\bb{S}_{\ll}\mc{R}_1) = 0$ or $H^b(\bb{P}V_2,\mc{Q}_2^{2r+v}\oo\bb{S}_{\mu}\mc{R}_2) = 0$. Suppose otherwise, so that both groups are non-zero. It follows from Bott's theorem that $\ll_a-1\geq 2r+a\geq 2+a$ and $\mu_b-1\geq 2r+v+b\geq 2+v+b$, so
 \begin{equation}\label{eq:ineq-2v}
 2v = |\ll| \geq \ll_1+\cdots+\ll_a \geq a\cdot\ll_a \geq a(3+a),
 \end{equation}
 \begin{equation}\label{eq:ineq-2u+v}
2u+v = |\mu| \geq \mu_1+\cdots+\mu_b \geq b\cdot\mu_b \geq b(3+b+v).
\end{equation}
We divide our analysis into three cases and show that in each case we obtain a contradiction.
 
\noindent\emph{Case 1: $b=0$.} It follows that $a=j$, so that $j+2\geq u+v\geq v\geq j(j+3)/2$ by (\ref{eq:ineq-2v}), which is only possible when $j=1$, which then forces $3\geq v\geq 2$. Since $\bb{S}_{\b}\circ\bb{S}_{1,1}$ is a quotient of $\bb{S}_{1,1}^{\oo v}$, it follows by Pieri's rule that $\ll_1\leq v\leq 3$, but this contradicts the inequality $\ll_a-1\geq 2+a$ since $a=j=1$.
 
 \noindent\emph{Case 2: $a=0$.} It follows that $b=j$, so that
 \[ 2(j+2)\geq 2(u+v) = (2u+v)+v \overset{(\ref{eq:ineq-2u+v})}{\geq} j(3+j+v) + v=3j+j^2+(j+1)v.\] 
 It follows that $4\geq j+j^2+(j+1)v$, which implies $j=1$, so that $u+v\leq 3$. Since $\bb{S}_{\a'}\circ\bb{S}_{1,1}$ is a quotient of $\bb{S}_{1,1}^{\oo u}$ and $\bb{S}_{\b'}$ is a quotient of $\bb{S}_{1,1}^{\oo v}$, it follows from Pieri's rule that $\mu_1\leq u+v\leq 3$, but this contradicts the inequality $\mu_b\geq b(3+b+v)\geq b(3+b)$ since $b=j=1$.

\noindent\emph{Case 3: $a,b\geq 1$.} It follows that
 \[2(a+b)+4=2(j+2)\geq 2v+(2u+v)-v \overset{(\ref{eq:ineq-2v})-(\ref{eq:ineq-2u+v})}{\geq} a(3+a)+b(3+b+v) - v=3(a+b)+a^2+b^2+(b-1)v,\]
 which implies that $4\geq a+b+a^2+b^2$, forcing $a=b=1$ and equality to hold everywhere. In particular, $2v=a(3+a)$ implies $v=2$, and $2u+v=b(3+b+v)$ implies $u=2$. Since (\ref{eq:ineq-2v})--(\ref{eq:ineq-2u+v}) are equalities, we get moreover that $\ll=(4)$ and $\mu=(6)$ are partitions with only one part. However, every $\ll$ for which $\bb{S}_{\ll}$ appears in $\bb{S}_{\b}\circ\bb{S}_{1,1}$ has at least two parts, so we reached once again a contradiction.
\end{proof}

\begin{proposition}\label{prop:geometric-Mr-Mr-1}
If $r>0$ then $M_r/M_{r-1} = H^0(X,\mc{L}^{2r}\oo\Sym(\eta))$, and $H^i(X,\mc{L}^{2r}\oo\Sym(\eta))=0$ for $i>0$.
 Moreover, for $i=0,1,2$ we have that $\Tor_i^R(M_r/M_{r-1},\bb{C})_j = 0\mbox{ if }j\neq i+r$, so Theorem~\ref{thm:Tor-grM} holds.
\end{proposition}

\begin{proof}
Denote $N_r = H^0(X,\mc{L}^{2r}\oo\Sym(\eta))$.
 Note that since $\mc{Q}_i$ is a line bundle, we have that for a partition $\mu\vdash d$ one has
 \[\bb{S}_{\mu}(\mc{R}_i\oo\mc{Q}_i) = \mc{Q}_i^d \oo \bb{S}_{\mu}\mc{R}_i.\]
 By Cauchy's formula, one gets
 \[\mc{L}^{2r}\oo\Sym^d(\eta) = \bigoplus_{\mu\vdash d}\pi_1^*(\mc{Q}_1^{d+2r} \oo \bb{S}_{\mu}\mc{R}_1) \oo \pi_2^*(\mc{Q}_2^{d+2r} \oo \bb{S}_{\mu}\mc{R}_2).\]
Using K\"unneth's formula, we get that
\[ H^t(X,\mc{L}^{2r}\oo\Sym^d(\eta)) = \bigoplus_{\mu\vdash d}\left(\bigoplus_{u+v=t} H^u(\bb{P}V_1,\mc{Q}_1^{d+2r} \oo \bb{S}_{\mu}\mc{R}_1) \oo H^v(\bb{P}V_2,\mc{Q}_2^{d+2r} \oo \bb{S}_{\mu}\mc{R}_2)\right).\]
Since $|\mu|=d$ and $r\geq 0$, it follows that $d+2r\geq\mu_1$, so Bott's theorem implies that the sheaves $\mc{Q}_i^{d+2r} \oo \bb{S}_{\mu}\mc{R}_i$ have vanishing higher cohomology. It follows that $\mc{L}^{2r}\oo\Sym^d(\eta)$ has vanishing higher cohomology, and 
\begin{align*}
H^0(X,\mc{L}^{2r}\oo\Sym^d(\eta)) =& \bigoplus_{\mu\vdash d} H^0(\bb{P}V_1,\mc{Q}_1^{d+2r} \oo \bb{S}_{\mu}\mc{R}_1) \oo H^0(\bb{P}V_2,\mc{Q}_2^{d+2r} \oo \bb{S}_{\mu}\mc{R}_2) \\
\overset{\text{Bott}}{=}& \bigoplus_{\mu\vdash d} \bb{S}_{d+2r,\mu_1,\mu_2,\cdots}V_1 \oo \bb{S}_{d+2r,\mu_1,\mu_2,\cdots}V_2.
\end{align*}
Since the partitions in $\mc{M}_r\setminus\mc{M}_{r-1}$ are precisely the ones of the form $(d+2r,\mu_1,\mu_2,\cdots)$ with $|\mu|=d$ for some $d\geq 0$, we can then apply Lemma~\ref{lem:dec-Mr-Mr-1} to conclude that $M_r/M_{r-1} \cong  N_r$ as $\GL$-representations.

Using \cite[Theorem~5.1.2]{weyman} it follows  that
\begin{equation}\label{eq:Tori-Mr-Mr-1}
\Tor_i^R(N_r,\bb{C})_{r+i+j} = H^j\left(X,\bw^{i+j}\xi\oo\mc{L}^{2r}\right)\mbox{ for all }i,j\in\bb{Z}.
\end{equation}
For $i\leq 2$, it follows from \cite[Exercise~II.5.16]{hartshorne} that $\bw^{i+j}\xi$ has a filtration with composition factors $\bw^u\xi_1\oo\bw^v\xi_2$, with $u+v=i+j\leq j+2$, so using Lemma~\ref{lem:vanishing-xi1-xi2} we get that $H^j\left(X,\bw^{i+j}\xi\oo\mc{L}^{2r}\right)=0$ when $j,r\geq 1$. 
In particular,
 from (\ref{eq:Tori-Mr-Mr-1}) we get 
 $ \Tor_i^R(N_r,\bb{C})_j = 0\mbox{ if }j\neq i+r$
 for $i=0,1,2$.

To conclude our proof we show that $M_r/M_{r-1} \cong  N_r$ as $R$-modules. By the previous paragraph we have
\[
\Tor_0^R(N_r, \bb{C}) = H^0(X, \mc{L}^{2r}) = \Sym^{2r} V_1 \oo \Sym^{2r} V_2
\mbox{ and }
\Tor_1^R(N_r, \bb{C}) =  H^0(X, \mc{L}^{2r}\oo \xi).
\]
To compute the latter, notice that the exact sequence $0 \rightarrow \xi_1 \rightarrow \xi \rightarrow \xi_2 \rightarrow 0$ yields an exact sequence
$$
0 \lra H^0\left(X, \mc{L}^{2r} \oo \xi_1\right) \lra
H^0\left(X, \mc{L}^{2r} \oo \xi\right) \lra
H^0\left(X, \mc{L}^{2r} \oo \xi_2\right).
$$
Using K\"unneth's formula, Bott's theorem, and Pieri's rule we calculate
\begin{align*}
H^0\left(X, \mc{L}^{2r} \oo \xi_1\right)  = &
H^0\left(X,
 \pi_1^* \mc{Q}_1^{2r}\oo
 \pi_2^* \mc{Q}_2^{2r}\oo
 \bigwedge^2 V_1 \oo \pi^*_2  \left(\bigwedge^2 \mc{R}_2\right)\right)
 \\
= &
 \bigwedge^2 V_1 \oo
H^0\left(X,
 \pi_1^* \mc{Q}_1^{2r}\oo
 \pi^*_2  \left(\mc{Q}_2^{2r}\oo\bigwedge^2 \mc{R}_2\right)\right)
\\
\overset{\text{K\"unneth}}{=} &
 \bigwedge^2 V_1 \oo
H^0\left(\bb{P} V_1, 
 \mc{Q}_1^{2r}\right)
 \oo 
H^0\left(\bb{P} V_2, 
 \mc{Q}_2^{2r}\oo
  \bigwedge^2 \mc{R}_2\right)
\\
 \overset{\text{Bott}}{=} &
\left( \bigwedge^2 V_1 \oo
\Sym^{2r} V_1 \right)\oo 
\bb{S}_{2r,1,1} V_2
\\
\overset{\text{Pieri}}{=} &
\left(
\bb{S}_{2r,1,1} V_1
\oplus
\bb{S}_{2r+1,1} V_1
\right)
\oo 
\bb{S}_{2r,1,1} V_2,
\\
H^0\left(X, \mc{L}^{2r} \oo \xi_2\right)  =&
 H^0\left(X,
 \pi_1^* \mc{Q}_1^{2r}\oo
 \pi_2^* \mc{Q}_2^{2r}\oo
\pi^*_1  \left(\bigwedge^2 \mc{R}_1\right) \oo \pi^*_2  (\mc{R}_2\oo \mc{Q}_2)\right)
 \\
 = &
H^0\left(X, 
\pi^*_1  \left(\mc{Q}_1^{2r}\oo\bigwedge^2 \mc{R}_1\right) \oo \pi^*_2  (\mc{Q}_2^{2r}\oo\mc{R}_2\oo \mc{Q}_2)\right)
 \\
 \overset{\text{K\"unneth}}{=}  &
H^0\left(\bb{P} V_1, 
 \mc{Q}_1^{2r} \oo \bigwedge^2 \mc{R}_1
  \right)
 \oo 
H^0\left(\bb{P} V_2, 
 \mc{Q}_2^{2r+1}\oo
 \mc{R}_2\right)\\
\overset{\text{Bott}}{=}  & 
\bb{S}_{2r,1,1} V_1
\oo 
\bb{S}_{2r+1,1} V_2.
\end{align*}
We deduce that
$\Tor_1^R(N_r, \bb{C})$ is a subrepresentation of 
\begin{equation}\label{eq:Tor1-Nr}
\bb{S}_{2r,1,1} V_1 \oo \bb{S}_{2r,1,1} V_2 
\oplus
\bb{S}_{2r,1,1} V_1 \oo \bb{S}_{2r+1,1} V_2
\oplus
\bb{S}_{2r+1,1} V_1 \oo \bb{S}_{2r,1,1} V_2.
\end{equation}
Since both $M_r/M_{r-1}$ and $N_r$ are generated by $\Sym^{2r}V_1\oo \Sym^{2r}V_1$,
there exist  $R$-linear surjections
\[
\varphi_M : F \twoheadrightarrow M_r/M_{r-1},
\qquad
\varphi_N : F \twoheadrightarrow N_r
\]
where $F = \Sym^{2r}V_1\oo \Sym^{2r}V_1 \oo R$.
Let $H\subseteq F$ be the $R$-submodule generated by all subspaces
$\bb{S}_\ll V_1 \oo \bb{S}_\mu V_2$ such that either $\ll \ne \mu$ or $\ll = \mu \notin \mc{M}_r\setminus\mc{M}_{r-1}$.
It follows by Lemma \ref{lem:dec-Mr-Mr-1} that $H\subseteq \ker(\varphi_M)$.
On the other hand, by (\ref{eq:Tor1-Nr}) we see that  $\ker(\varphi_{N}) \subseteq H$,
since $(2r,1,1) \in \mc{M}_{r-1}$.
Hence $\varphi_M, \varphi_N$ induce an $R$-linear surjection  $ N_r \twoheadrightarrow M_r/M_{r-1}$,
which must also be injective since 
$N_r \cong M_r/M_{r-1}$ as  $\GL$-representations.
\end{proof}

\section{The subspace variety}\label{sec:subspace}

Recall that $\ol{W} = \Sym^2 V_1\oo \Sym^2 V_2$, and $\ol{R} = \Sym(\ol{W})$, so that $\op{Spec}(\ol{R}) = \ol{W}^{\vee}$. We consider the \defi{subspace variety $Y\subset\ol{W}^{\vee}$} given by
\[ Y = \{ y \in \ol{W}^{\vee} : y\in \Sym^2 H \oo \Sym^2 V_2^{\vee}\mbox{ for some codimension one subspace }H\subset V_1^{\vee}\}.\] 
The goal of this section is to describe the defining equations of $Y$.

\begin{theorem}\label{thm:subspace}
 The vector space of minimal generators of the ideal $I(Y)\subset\ol{R}$ is concentrated in degree $m$, and it is isomorphic as a $\GL$-representation to
 \[\bw^m V_1 \oo \bw^m(V_1\oo\Sym^2 V_2).\]
\end{theorem}

The proof techniques are similar to those used in \cite[Chapter~7]{weyman}. We consider the projective space $\PP$ of  rank $(m-1)$ quotients of $V_1$, and consider the affine bundle
\[\mc{Y} = \ul{\op{Spec}}_{\PP}\Sym(\Sym^2\mc{Q}\oo \Sym^2 V_2),\]
where $\mc{Q}$ is the tautological quotient sheaf of rank $m-1$ (note that this is different from (\ref{eq:tautological-ses}), where the rank of $\mc{Q}_1$ was one). The tautological sequence on $\PP$ is
\[0\lra\mc{R} \lra V_1 \oo \mc{O}_{\PP} \lra \mc{Q} \lra 0,\]
where $\mc{R}$ is often denoted $\mc{O}_{\PP}(-1)$. It induces an exact sequence
\[0\lra\mc{R}\oo V_1 \lra \Sym^2V_1 \oo \mc{O}_{\PP} \lra \Sym^2\mc{Q} \lra 0,\]
and after tensoring with $\Sym^2 V_2$ it gives an exact sequence
\[0\lra(\mc{R}\oo V_1)\oo \Sym^2 V_2\lra \ol{W} \oo \mc{O}_{\PP} \lra \Sym^2\mc{Q} \oo \Sym^2 V_2 \lra 0,\]
which makes $\mc{Y}$ a geometric sub-bundle of the trivial bundle $\ol{W}^{\vee} \times \PP$. Writing $q:\ol{W}^{\vee} \times \PP \lra \ol{W}^{\vee}$ for the natural projection map, we have that $q(\mc{Y}) = Y$ (in fact, the reader can check that $q_{|_{\mc{Y}}}$ is birational, so it gives a resolution of singularities of $Y$). Using Bott's theorem and the projection formula, we have that
\[ H^i(\mc{Y},\mc{O}_{\mc{Y}}) = \bigoplus_{d\geq 0}H^i(\PP,\Sym^d(\Sym^2\mc{Q} \oo \Sym^2 V_2))=0\mbox{ for }i>0.\]
It follows from \cite[Theorem~5.1.2]{weyman} that $q_*(\mc{O}_{\mc{Y}})$ has a minimal free $\ol{R}$-resolution $F_{\bullet}$ where
\[ 
\begin{aligned}
F_i &= \bigoplus_j H^j\left(\PP,\bw^{i+j}(\mc{R}\oo V_1\oo \Sym^2 V_2)\right)\oo\ol{R}(-i-j) \\
&=\bigoplus_j H^j(\PP,\mc{R}^{i+j})\oo\bw^{i+j}(V_1\oo \Sym^2 V_2) \oo\ol{R}(-i-j).\\
\end{aligned}
\]
For $i+j\geq 0$, we have that $H^j(\PP,\mc{R}^{i+j})=H^j(\PP,\mc{O}_{\PP}(-i-j))$ is non-zero only for $i=j=0$, and $j=m-1$, $i+j\geq m$. In particular, we get that
\[ F_1 = H^{m-1}(\PP,\mc{R}^m) \oo \bw^{m}(V_1\oo \Sym^2 V_2) \oo \ol{R}(-m).\]
Since the minimal generators of $F_1$ are the minimal generators of $I(Y)$, and since $H^{m-1}(\PP,\mc{R}^m)=\bw^m V_1$, the conclusion of Theorem~\ref{thm:subspace} follows.

\section{The proof of Theorem~\ref{thm:relations-ol-A}}\label{sec:induction}

The goal of this section is to describe the proof of Theorem~\ref{thm:relations-ol-A}. We let $(m,n)=(\dim(V_1),\dim(V_2))$, and note that by symmetry, the cases $(m,n)$ and $(n,m)$ are equivalent, so without loss of generality we may assume that $m\geq n$. In Section~\ref{subsec:low-degree} we use the results of \cite{BCV} to prove Theorem~\ref{thm:relations-ol-A} for $j\leq 4$. Based on Corollary~\ref{cor:Tor1-low-degree}, we deduce that the theorem is true for $n\leq 5$ in Section~\ref{subsec:n-leq-5}. The substantial part of the argument is explained in Section~\ref{subsec:inductive-step}, where we argue by induction on the dimension vector $(m,n)$.

\subsection{Low degree equations}\label{subsec:low-degree}
It follows from transpose duality that proving Theorem~\ref{thm:relations-ol-A} for $j\leq 4$ (for all $m,n$) is equivalent to proving Theorem~\ref{thm:relations-A} for $j\leq 4$ (for all $m,n$). Since the minors (resp. permanents) are linearly independent, we may assume that $j\geq 2$. The cases $j=2,3$ and $4$ are discussed in Sections~2.1, 3.3, and 3.4 (respectively) of \cite{BCV}.

\subsection{The case $n\leq 5$}\label{subsec:n-leq-5}
The assumption $n\leq 5$ implies that $\lfloor n/2\rfloor+2\leq 4$, so we conclude by Corollary~\ref{cor:Tor1-low-degree} that $\Tor_1^{\ol{R}}(\ol{A},\bb{C})_j = 0$ for $j>4$.

\subsection{The inductive step}\label{subsec:inductive-step}
We assume that $m\geq n\geq 6$, and suppose by induction that Theorem~\ref{thm:relations-ol-A} is true for every pair $(m',n')\neq(m,n)$, with $m'\leq m$ and $n'\leq n$ (we abbreviate this as $(m',n')\prec(m,n)$). If we define the functor $\ol{\mf{T}}=\ol{\mf{T}}_2 \oplus \ol{\mf{T}}_3$ by letting (recall the notation and conventions in Section~\ref{subsec:bi-functors})
\[\ol{\mf{T}}_2 = \bb{S}_{4}\boxtimes\bb{S}_{2,2} \oplus \bb{S}_{2,2}\boxtimes \bb{S}_{4},\quad \ol{\mf{T}}_3 = \bb{S}_{4,1,1}\boxtimes\bb{S}_{3,3} \oplus \bb{S}_{3,3}\boxtimes \bb{S}_{4,1,1},\]
then the induction hypothesis gives for every $U_1,U_2$ with $(\dim(U_1),\dim(U_2))\prec(m,n)$ an exact sequence
\[  \ol{\mf{T}}(U_1,U_2) \oo \ol{\mf{R}}(U_1,U_2) \lra \ol{\mf{R}}(U_1,U_2) \lra \ol{\mf{A}}(U_1,U_2) \lra 0.\]
By functoriality, we can extend $\ol{\mf{T}},\ol{\mf{R}},\ol{\mf{A}}$ to bi-variate functors of locally free sheaves on any variety (or scheme) $Z$ over $\bb{C}$. Since exactness is a local property, it follows that 
\begin{equation}\label{eq:relative-presentation}
  \ol{\mf{T}}(\mc{V}_1,\mc{V}_2) \oo_{\mc{O}_Z} \ol{\mf{R}}(\mc{V}_1,\mc{V}_2) \lra \ol{\mf{R}}(\mc{V}_1,\mc{V}_2) \lra \ol{\mf{A}}(\mc{V}_1,\mc{V}_2) \lra 0
\end{equation}
is exact for any pair of locally-free sheaves $(\mc{V}_1,\mc{V}_2)$ with $(\rk\mc{V}_1,\rk\mc{V}_2)<(m,n)$.

Consider now vector spaces $V_1,V_2$ with $(\dim(V_1),\dim(V_2))=(m,n)$, and consider $Z = \PP$ the projective space of $(m-1)$-dimensional quotients of $V_1$ (as in Section~\ref{sec:subspace}). We take $\mc{V}_1 = \mc{Q}$ the tautological rank $(m-1)$ quotient sheaf on $\PP$, and $\mc{V}_2 = V_2 \oo \mc{O}_{\PP}$. Since $(\rk\mc{V}_1,\rk\mc{V}_2)=(m-1,n)\prec(m,n)$, (\ref{eq:relative-presentation}) is exact. We let
\[ B = H^0(\PP,\ol{\mf{R}}(\mc{V}_1,\mc{V}_2))\]
and observe that (using the notation from Section~\ref{sec:subspace}) $B=\ol{R}/I(Y)$ is the coordinate ring of the subspace variety $Y\subset\ol{W}^{\vee}$. It follows from Theorem~\ref{thm:subspace} that $B$ has an $\ol{R}$-module presentation
\begin{equation}\label{eq:olR-presentation-B}
 \mf{U}(V_1,V_2) \oo \ol{R} \to \ol{R} \onto B,\mbox{ where }\mf{U}(V_1,V_2) = \bw^m V_1 \oo \bw^m(V_1\oo\Sym^2 V_2),
\end{equation}
so $B$ is defined by degree $m$ equations in $\ol{R}$. We let
\begin{equation}\label{eq:charC}
C =  H^0(\PP,\ol{\mf{A}}(\mc{V}_1,\mc{V}_2)),
\end{equation}
and note that since the maps in (\ref{eq:relative-presentation}) are split as maps of $\mc{O}_{Z}$-modules, it follows that after taking global sections we obtain an exact sequence
\begin{equation}\label{eq:res-TR-B-C}
 H^0\left(\PP,\ol{\mf{T}}(\mc{V}_1,\mc{V}_2) \oo_{\mc{O}_Z} \ol{\mf{R}}(\mc{V}_1,\mc{V}_2)\right) \lra B \lra C\lra 0.
\end{equation}
We claim moreover that the natural multiplication map
\begin{equation}\label{eq:T-oo-R-onto}
 H^0\left(\PP,\ol{\mf{T}}(\mc{V}_1,\mc{V}_2)\right) \oo H^0\left(\PP,\ol{\mf{R}}(\mc{V}_1,\mc{V}_2)\right) \lra H^0\left(\PP,\ol{\mf{T}}(\mc{V}_1,\mc{V}_2) \oo_{\mc{O}_Z} \ol{\mf{R}}(\mc{V}_1,\mc{V}_2)\right))
\end{equation}
is also surjective. To see this, it suffices by \cite[Example~1.8.13]{lazarsfeld} to check that $\ol{\mf{T}}(\mc{V}_1,\mc{V}_2)$ and $\ol{\mf{R}}(\mc{V}_1,\mc{V}_2)$ are (direct sums of coherent locally free) $0$-regular sheaves. Since $\mc{V}_1$ is resolved by the complex 
\[0\lra\mc{O}_{\PP}(-1)\lra V_1 \oo \mc{O}_{\PP}[\lra \mc{V}_1\lra 0],\]
it follows from \cite[Proposition~1.8.8]{lazarsfeld} that $\mc{V}_1$ is $0$-regular, and by \cite[Proposition~1.8.9]{lazarsfeld} the same is true about any polynomial functor applied to $\mc{V}_1$. The same reasoning applies to $\mc{V}_2$, which is $0$-regular since it is trivial. Since tensor products of $0$-regular locally free sheaves are $0$-regular, the desired claim follows.

Combining (\ref{eq:res-TR-B-C}) with the surjective map (\ref{eq:T-oo-R-onto}), and noting that by Bott's theorem we have 
\[H^0\left(\PP,\ol{\mf{T}}(\mc{V}_1,\mc{V}_2)\right)=\ol{\mf{T}}(V_1,V_2),\]
we obtain an exact sequence
\[\ol{\mf{T}}(V_1,V_2) \oo B \lra B \lra C \lra 0.\]
Using the presentation (\ref{eq:olR-presentation-B}) of $B$ as an $\ol{R}$-module, we obtain a presentation of $C$ given by
\begin{equation}\label{eq:presentation-C}
 \left(\ol{\mf{T}}(V_1,V_2) \oplus \mf{U}(V_1,V_2)\right) \oo \ol{R} \lra \ol{R} \onto C,
\end{equation}
which shows that 
\begin{equation}\label{eq:ub-Tor1-C}
 \Tor_1^{\ol{R}}(C,\bb{C}) \subseteq \ol{\mf{T}}(V_1,V_2) \oplus \mf{U}(V_1,V_2).
\end{equation}

If $m>n$ then we have that $C = \ol{A}$, so (\ref{eq:ub-Tor1-C}) yields
\[ \Tor_1^{\ol{R}}(\ol{A},\bb{C})_j = 0\mbox{ for }j\neq 2,3,m.\]
Using the fact that $m\geq n>\lfloor n/2\rfloor + 2$ and Corollary~\ref{cor:Tor1-low-degree} we conclude that $\Tor_1^{\ol{R}}(\ol{A},\bb{C})_m=0$, and therefore $\Tor_1^{\ol{R}}(\ol{A},\bb{C})_j=0$ for $j>4$.

Suppose now that $m=n$. In this case $C$ is only a quotient of $\ol{A}$, and we have a short exact sequence
\[ 0 \lra J \lra \ol{A} \lra C \lra 0,\]
where $J$ is an ideal whose $\GL$-structure is given (applying Bott's theorem to (\ref{eq:charC})) by
\[ J = \bigoplus_{\substack{\ll\in\ol{\mc{M}}_0 \\ \ll_n\neq 0}}\bb{S}_{\ll}V_1 \oo \bb{S}_{\ll}V_2,\]
where $\ol{\mc{M}}_0 = \{\ll' : \ll\in\mc{M}_0\}$. Since partitions with $n$ parts only occur in degree $\geq n$ in $\ol{R}$, it follows that $J$ is generated in degree $\geq n$, so its defining relations have degree $\geq n+1$ and thus $\Tor_1^{\ol{R}}(J,\bb{C})_j=0$ for $j\leq n$. We consider the long exact sequence on $\Tor$ in degree $j$:
\[ \cdots \lra \Tor_1^{\ol{R}}(J,\bb{C})_j \lra \Tor_1^{\ol{R}}(\ol{A},\bb{C})_j \lra \Tor_1^{\ol{R}}(C,\bb{C})_j \lra \cdots\]
By (\ref{eq:ub-Tor1-C}) we know that if $j>4$ then $\Tor_1^{\ol{R}}(C,\bb{C})_j$ may only be non-zero for $j=n>\lfloor n/2\rfloor + 2$, so it doesn't contribute to $\Tor_1^{\ol{R}}(\ol{A},\bb{C})_j$ by Corollary~\ref{cor:Tor1-low-degree}. A similar argument applies to $\Tor_1^{\ol{R}}(J,\bb{C})_j$, which may be non-zero only for $j\geq n+1$. This proves that $\Tor_1^{\ol{R}}(\ol{A},\bb{C})_j=0$ for $j>4$, concluding our proof.

\section{Determinantal ideals of fiber type}\label{sec:Rees}

In this final section we turn our attention to the   graph of the map $\Lambda_2$ introduced in \eqref{eq:phi-wedgephi},
 which is also the blowup of 
$\bb{P}(\Hom(\bb{C}^m,\bb{C}^n))$ along the determinantal variety defined by $I_2$.
Observe that
\[
\Graph\left(\Lambda_2\right) \subseteq 
\bb{P}(\Hom(\bb{C}^m,\bb{C}^n)) \times \bb{P}\left(\Hom\left(\bw^2\bb{C}^m,\bw^2\bb{C}^n\right)\right)
\]
thus $\Graph\left(\Lambda_2\right)$ is defined by bi-homogeneous polynomials in 
the ring $S \otimes R$. 
In fact, the bi-homogeneous coordinate ring of  $\Graph\left(\Lambda_2\right)$ is  the \defi{Rees algebra} of the ideal of minors $I_2 \subseteq S$
\[
\Rees(I_2) = \bigoplus_{d\geq 0} I_2^d \cong S[It] \subseteq S [t]
\] 
where $t$ denotes an indeterminate.
The ring  $\Rees(I_2)$ is the image of the $S$-algebra map $\Pi: S\oo R \rightarrow S[t]$
determined by  $R_1 = W \rightarrow W t \subseteq S_2 t$.
In this section we scale the grading of $R$ and consider it as polynomial ring generated in degree $2$.
In particular, $\Rees(I_2)$
 is a bi-graded $\bb{C}$-algebra  generated  by
$\Rees(I_2)_{(1,0)} = V_1\oo V_2 = S_1$ and $\Rees(I_2)_{(0,2)} = W t\subseteq S_2 t$.

It is natural to study the defining relations of $\Rees(I_2)$.
From this perspective, the algebra $A$, investigated in the previous sections,  is the \defi{special fiber} $\Rees(I_2) \oo_S \bb{C}$ of the Rees algebra.
Denoting by $\mc{J} = \ker(\Pi)$ and $I(X_{m,n}) = \ker (\Psi)$ the defining ideals of $\Rees(I_2)$ and $A$ respectively,
we have that $\mc{J}_{(0,2d)} = I(X_{m,n})_d $ for all $d\in \bb{N}$.
On the other hand, denoting by $\op{Syz}(I_2)$ the first syzygy module of $I_2$,
we have that $\mc{J}_{(d,2)} = \op{Syz}(I_2)_{d+2}$ for all $d\in \bb{N}$.
The ideal $I_2$ is said to be of \defi{fiber type} if $\mc{J}$ is generated by 
 $\mc{J}_{(0,2d)} $ and $\mc{J}_{(d,2)}$ for $d \in \bb{N}$.
In analogy with the case of maximal minors  (see \cite{BCV2}) we ask:

\begin{question}\label{que:fibert-type}
Let $I_2$ be the ideal   of $2 \times 2 $ minors of a generic $m\times n$ matrix $X$.
Is $I_2$ of fiber type?
\end{question}

The fiber type property would reduce the problem of finding the relations of $\Rees(I_2)$ to the one of finding the first syzygies of $I_2$, solved by Lascoux, and the one of finding the relations of $A$, settled in Theorem \ref{thm:relations-A}.
Below we observe,
adapting an argument of \cite{BCV},
 that Question \ref{que:fibert-type} reduces to matrices of size $(n+2) \times n$, and that the  answer is affirmative for $m \times 3$ matrices.
 Recall our convention that $m = \dim V_1 \geq n = \dim V_2$.

\begin{proposition}\label{prop:fiber-type}
Fix $n\in \bb{N}$.
If $I_2$ is of fiber type for a generic $(n+2) \times n$ matrix, 
then it is of fiber type for any generic $m \times n$ matrix.
\end{proposition}
\begin{proof}
Denote $T = S\oo R$.
For each bi-degree $(d, e) \ne (0,0)$ we consider the surjective map
 \[
\chi: D =  T_{(1,0)} \otimes T_{(d-1,e)} \oplus  T_{(0,2)} \otimes T_{(d,e-2)} \rightarrow T_{(d,e)}
 \]
 defined by $\chi(x\otimes \alpha, y\otimes \beta ) = x\alpha + y\beta$.
We say that a representation $\bb{S}_\ll V_1 \oo \bb{S}_\mu V_2 $ is balanced if $\ll = \mu$, unbalanced otherwise. 
Decompose $T_{(d,e)} =B_{(d,e)} \oplus U_{(d,e)}$,
 where
$ B_{(d,e)}$ (resp. $U_{(d,e)}$) is the sum of all the balanced (resp. unbalanced) sub-representations.
Furthermore, let $D = E\oplus F$ where 
\[
 E =T_{(1,0)} \otimes B_{(d-1,e)} \oplus  T_{(0,2)} \otimes B_{(d,e-2)},
 \qquad 
 F = T_{(1,0)} \otimes U_{(d-1,e)} \oplus  T_{(0,2)} \otimes U_{(d,e-2)}.
 \]
Notice that, since $\Rees(I_2)$ is a direct sum of ideals of $S$, 
it contains no unbalanced representation, and therefore $U_{(d,e)} \subseteq \mc{J}_{(d,e)}$.

Let $H\subseteq \mc{J}_{(d,e)}$ be an irreducible representation that 
is part of the minimal generating set of $\mc{J}$, i.e. such that its image   modulo $(T_{(1,0)}\oplus T_{(0,2)})\mc{J}$ is non-zero. 
Setting  $H'=\chi(F)$,
   we cannot have $H\subseteq H'$, since $H$ is  minimal.
Thus  $H\cap H' = 0$, and by surjectivity of $\chi$
there exists an irreducible sub-representation $H'' \subseteq D$, disjiont from $F$, mapping non-trivially to $H$ via $\chi$.
It follows that $H''$, and hence also $H$, is isomorphic to an irreducible representation occurring in $E$,
and hence occurring in either 
$ V_1 \oo V_2  \oo \bb{S}_\ll V_1 \oo \bb{S}_\ll V_2$ or 
$ \bigwedge^2V_1 \oo \bigwedge^2V_2  \oo \bb{S}_\ll V_1 \oo \bb{S}_\ll V_2$,
for some $\ll \in \P$.
We  have $\lambda_{n+1}= 0$ since $H \ne 0$, and  by Pieri's rule 
$H \cong \bb{S}_{\mu}V_1 \otimes \bb{S}_{\nu} V_2$ 
for some $\mu, \nu \in \P$ with $\mu_{n+3}= \nu_{n+3}=0$.
We conclude that the relations $H \subseteq \mc{J}_{(d,e)}$ already appear in the case  $m = n+2$, and the desired statement follows.
\end{proof}

It can be verified using Macaulay2 \cite{GS} that for a generic $5\times 3$ matrix, the ideal $I_2$ is of fiber type. Applying Proposition~\ref{prop:fiber-type} with $n=3$ we obtain the following. 

\begin{corollary}
 For a generic $m\times 3$ matrix, the ideal $I_2$ is of fiber type.
\end{corollary}

\section*{Acknowledgements}
 Experiments with the computer algebra software Macaulay2 have provided numerous valuable insights. Perlman acknowledges the support of the National Science Foundation Graduate Research Fellowship under grant DGE-1313583. Polini acknowledges the support of the National Science Foundation grant DMS-1601865. Raicu acknowledges the support of the Alfred P. Sloan Foundation, and of the National Science Foundation Grant DMS-1901886.


\begin{bibdiv}
		\begin{biblist}

\bib{BCV}{article}{
   author={Bruns, Winfried},
   author={Conca, Aldo},
   author={Varbaro, Matteo},
   title={Relations between the minors of a generic matrix},
   journal={Adv. Math.},
   volume={244},
   date={2013},
   pages={171--206},
}

\bib{BCV2}{article}{
   author={Bruns, Winfried},
   author={Conca, Aldo},
   author={Varbaro, Matteo},
   title={Maximal minors and linear powers},
   journal={J. Reine Angew. Math.},
   volume={702},
   date={2015},
   pages={41--53},
}

\bib{DCEP}{article}{
   author={de Concini, Corrado},
   author={Eisenbud, David},
   author={Procesi, Claudio},
   title={Young diagrams and determinantal varieties},
   journal={Invent. Math.},
   volume={56},
   date={1980},
   number={2},
   pages={129--165},
}

\bib{eisenbud-CA}{book}{
   author={Eisenbud, David},
   title={Commutative algebra},
   series={Graduate Texts in Mathematics},
   volume={150},
   note={With a view toward algebraic geometry},
   publisher={Springer-Verlag, New York},
   date={1995},
   pages={xvi+785},
}

\bib{GS}{article}{
          author = {Grayson, Daniel R.},
          author = {Stillman, Michael E.},
          title = {Macaulay 2, a software system for research
                   in algebraic geometry},
          journal = {Available at \url{http://www.math.uiuc.edu/Macaulay2/}}
        }
        
\bib{hartshorne}{book}{
   author={Hartshorne, Robin},
   title={Algebraic geometry},
   note={Graduate Texts in Mathematics, No. 52},
   publisher={Springer-Verlag, New York-Heidelberg},
   date={1977},
   pages={xvi+496},
}

\bib{lazarsfeld}{book}{
   author={Lazarsfeld, Robert},
   title={Positivity in algebraic geometry. I},
   series={Ergebnisse der Mathematik und ihrer Grenzgebiete. 3. Folge. A
   Series of Modern Surveys in Mathematics [Results in Mathematics and
   Related Areas. 3rd Series. A Series of Modern Surveys in Mathematics]},
   volume={48},
   note={Classical setting: line bundles and linear series},
   publisher={Springer-Verlag, Berlin},
   date={2004},
   pages={xviii+387},
}

\bib{sam-snowden}{article}{
   author={Sam, Steven V},
   author={Snowden, Andrew},
   title={Introduction to twisted commutative algebras},
   journal = {arXiv},
   number = {1209.5122},
   date={2012}
}

\bib{sam-snowden-category}{article}{
   author={Sam, Steven V},
   author={Snowden, Andrew},
   title={GL-equivariant modules over polynomial rings in infinitely many
   variables},
   journal={Trans. Amer. Math. Soc.},
   volume={368},
   date={2016},
   number={2},
   pages={1097--1158},
}

\bib{weyman}{book}{
   author={Weyman, Jerzy},
   title={Cohomology of vector bundles and syzygies},
   series={Cambridge Tracts in Mathematics},
   volume={149},
   publisher={Cambridge University Press, Cambridge},
   date={2003},
   pages={xiv+371},
}

		\end{biblist}
	\end{bibdiv}
\end{document}